\documentclass[a4paper,11pt,final]{article}
\usepackage{amsmath, amsthm, amssymb}
\usepackage{latexsym}
\usepackage[latin1]{inputenc}
\usepackage{epsf,exscale,times}
\usepackage[T1]{fontenc}
\usepackage[dvips]{graphicx}
\usepackage{graphics,epsfig}
\usepackage{epsfig,rotate,color}
\usepackage{showkeys}
\usepackage{subfigure}  

\newcommand{\SR}{\mathcal{S}(\mathbb{R})}

\newcommand{\I}{\mathcal{I}}

\newcommand{\F}{\mathcal{F}}
\newcommand{\R}{\mathbb{R}}

\newtheorem{lemme}{Lemma}
\newtheorem{definition}{Definition}
\newtheorem{proposition}{Proposition}

\newtheorem{remark}{Remark}

\title{ Finite difference approximations for a fractional diffusion/anti-diffusion equation}

\author{ Pascal Azerad $^*$ \& Afaf Bouharguane \footnote{Institut de Math\'ematiques et Mod\'elisation de Montpellier, UMR 5149 CNRS, Universit\'e Montpellier 2, Place Eug\`ene Bataillon, CC 051
34095 Montpellier, France. Emails:\, {\sffamily azerad@math.univ-montp2.fr}, {\sffamily bouharg@math.univ-montp2.fr}
The authors are supported by the ANR MATHOCEAN ANR-08-BLAN-0301-02. 
} 
}
\date{\today}

\begin{document}
\maketitle

\begin{abstract}
A class of finite difference schemes for solving a fractional anti-diffusive equation, recently proposed by Andrew C. Fowler to describe the dynamics of dunes, is considered. Their linear stability is analyzed using the standard Von Neumann analysis: stability criteria are found and checked numerically. Moreover, we investigate the consistency and convergence of these schemes. \\ 

\end{abstract}



\section{Introduction}

Partial Differential Equations with nonlocal or fractional operators are widely used to model scientific problems in mechanics, physics, signal processing, see for example \cite{signal} and references therein. We consider in this chapter a nonlocal conservation law which appears in the formation and dynamics of sand structures such as dunes and ripples \cite{fowler2,lagree}. 
Since it is generally impossible to obtain analytical solutions of these nonlocal models, one must rely on numerical solutions. In the last few decades, significant advances in numerical analysis and computational implementation of numerical methods for nonlocal/fractional PDEs have been made. For instance, \cite{JD2} 
 propose a finite volume method to approximate the solutions of a fractal scalar conservation law, that is to say a conservation law regularized by a \emph{diffusive} fractional power of the Laplacian operator and \cite{meerschaert,yuste} use finite difference methods to approximate fractional diffusive equations. \\
In this chapter, we develop the basic numerical analysis of the following evolution equation proposed by A.C. Fowler (see \cite{fowler2}, \cite{fowler3} and \cite{fowler1} for more details) to study the nonlinear dune formation:
\begin{equation}
\partial_t u(t,x) + \partial_x\left(\frac{u^2}{2}\right) (t,x) + \eta \, \I [u(t,\cdot)] (x)
- \epsilon \, \partial_{xx}^2 u(t,x) = 0, 
\label{fowlereqn4}
\end{equation}
where $u=u(t,x)$ represents the dune height and $\I$
is a nonlocal operator defined as follows: for any Schwartz
function $\varphi \in \SR$ and any $x \in \mathbb{R}$,
\begin{equation}
\I [\varphi] (x) := \int_{0}^{+\infty} |\xi|^{-\frac{1}{3}}
\varphi''(x-\xi) \, d\xi . \label{nonlocalterm3b}
\end{equation} 
The second and fourth terms of equation \eqref{fowlereqn4} correspond to the nonlinear and dissipative terms respectively, while the third term is the nonlocal term, which is anti-dissipative as we will show later on. The positive parameters $\epsilon$ (resp. $\eta$)  quantify the amount of local diffusion (resp. nonlocal anti-diffusion).

\begin{remark}
For causal functions (i.e. $\varphi(x) = 0 $ for $x<0$), this operator is, up to a multiplicative constant, the Riemann-Liouville integral which is defined as follows:
\begin{equation}
\frac{1}{\Gamma(\frac{2}{3})} \int_0^{+\infty} \frac{\varphi^{''}(x-\xi)}{|\xi|^{1/3}} d\xi= \frac{d^{-2/3}}{d x^{-2/3}} \varphi''(x) = 
\frac{d^{4/3}}{d x^{4/3}} \varphi(x),
\label{riemann}
\end{equation}
with $\Gamma$ the Euler function. 
\end{remark}
Many numerical methods for the evaluation of fractional order integrals and the solution of fractional order equations are proposed in the literature. Usually, time and spatial fractional derivatives are considered: we refer for instance to \cite{Diet,li, yuste}. \\
In our case, the integral operator $\I$ can be seen as a fractional power of order $2/3$ of the Laplacian with the bad sign. Indeed, it has been proved that $\I $ has the following Fourier transform \cite{alibaud}: 
\begin{equation}
\F(\I[\varphi]) (\xi) = \psi_{\I}(\xi) \F \varphi(\xi),
\label{fourier3bis}
\end{equation}
where $\psi_{\I}(\xi)=-a_{\I} |\xi|^{\frac{4}{3}}+ i b_{\I} \xi |\xi|^{\frac{1}{3}} $ with $a_{\I} = 2 \, \pi^2 \, \Gamma(\frac{2}{3})$, $b_{\I} = 2 \, \pi^2 \, \sqrt{3} \, \Gamma(\frac{2}{3}) $ and $\F$ denotes the Fourier transform defined for $f \in L^1(\R)$ by: for all $\xi \in \R$
\begin{equation*}
\F f(\xi) = \int_\R e^{-2i \pi x \xi } f(x) \, dx.
\end{equation*} 
Formula \eqref{fourier3bis} stems from the following integral formula \cite{alibaud}:
\begin{equation}
\I [\varphi](x)=  \frac{4}{9}\int_{-\infty}^{0}
\frac{\varphi(x+z)-\varphi(x)-\varphi'(x) z}{|z|^{7/3}} \: 
dz.\label{intformula}
\end{equation}
Finally, equation \eqref{fowlereqn4} involves two antagonistic terms: the anti-diffusive operator $\I$ which creates instabilities and the diffusion operator $-\partial_{xx}^2$ which controls these perturbations. \\
Recently, some theoretical results regarding the Fowler model \eqref{fowlereqn4}  have been obtained, namely, existence of travelling-waves, the global well-posedness, the failure of the maximum principle and the instability of constant solutions \cite{alibaud, alvarez, afaf}. The last two results are a consequence of the non-positivity of the kernel $K$ of $\I-\partial_{xx}^2$ defined for  $t > 0$ and $x \in \R$ by
\begin{equation}
K(t, \cdot)(x):= \F^{-1}(e^{-t(4 \pi^2 |\cdot|^2 + \psi_\I(\cdot)) })(x).
\label{kernel4}
\end{equation}
These two ``bad properties'' show that the discrete problem must be handled with care. Indeed, for monotone models, a classical way to get numerical stability criteria for explicit scheme is to ensure that the approximated problem satisfies the discrete maximum principle, which cannot be true for Equation \eqref{fowlereqn4}. \\
In \cite{alibaud}, some numerical results regarding this equation have been obtained using an explicit finite difference scheme but the detailed numerical study was not performed. Hence, in this chapter, we would like to go one step further investigating the numerical stability, consistency and convergence of a class of explicit finite difference schemes approximating the Fowler equation. \\
The numerical stability is specially interesting here because the growth of the solution depends on frequencies and time. Hence, the notion of $A$-stability, also called strong stability, is not suitable nor desirable. In the literature, some authors use another definition of stability, less restrictive than the $A$-stability: the $C$-stability. This is an abbreviation for convergence stability and is linked with stability in the Lax-Richtmyer sense. In this definition, a numerical scheme is stable for the norm $||\cdot||$ if for all $T>0$, there exists a constant $K(T)>0$ independent of the time and space steps $\delta x,\delta t$ such that for all initial data $u_0$   
\begin{equation*}
||u^n|| \leq K(T) ||u_0||, \hspace{0.5 cm} \forall \, 0 \leq n \leq \frac{T}{\delta t},
\end{equation*}
where $u^n$ represents the approximated solution at the time $t_n=n \delta t$. This definition allows the solution to grow with time, which is the case for example for the equation $u_t-u_{xx}=cu$. 
For the $L^2$-stability, a simple way to prove the numerical stability and specially to get stability criteria is  Fourier analysis, see Section \ref{stability}. Hence, considering the $C$-stability, the Von Neumann condition is written as
\begin{eqnarray}
&\,& \exists C >0, \exists \delta t^*>0, \mbox{ such that } \forall \delta t \in ]0,\delta t^*]; \forall k \in \mathbb{Z} \nonumber \\
&\,& \hspace{2.5 cm} |g(k)| \leq 1 + C \delta t, 
\label{vonNeu}
\end{eqnarray}
where $g$ is the discrete amplification factor, $k$ the wave number and $C$ is a positive constant independent of $\delta x$ and $\delta t$. If $C=0$, the Von Neumann condition coincides with the $A$-stability. \\ 
As we will see later in Section \ref{preliminaire}, the amplification of solutions of the Fowler equation also depends on frequencies: low frequencies are slowly amplified whereas the high frequencies are dampened. 
Hence, the notion of $C$-stability is not adapted for this model because it considers only the amplification due to time.
To take into account this phenomenon, the ``constant'' $C$ introduced in the Von Neumann condition \eqref{vonNeu} should also depend on the space step in order to be able to control the amplification w.r.t. different frequencies and this is not possible for a constant, by definition. 
Since high frequencies are usually responsible of numerical instabilities, we are going to focus our attention on them. Thereafter, the idea is to exhibit numerical stability conditions to ensure the validity of simulations. We then seek numerical stability criteria such that the amplification factor satisfies:
\begin{equation}
 \forall |k| \geq k_0, |g(k)| \leq 1,
\end{equation}
where $k_0$ is some threshold frequency.
To ensure this inequality, we will exhibit two sufficient conditions. The first one is rather unusual: it imposes to the space step $\delta x$ to be smaller than a given positive constant which depends on  the ratio 
$\epsilon/\eta$
of local  diffusion  to  non-local  anti-diffusion. We will in fact check numerically that this condition is not necessary. The second one looks more familiar. 
It is a classical CFL-type condition modified by a $\eta \, \delta t/\delta x^{4/3}$ term, which stems from the nonlocal operator. We will see in the numerical simulations that this condition is both necessary and sufficient to ensure numerical stability. \\
For a comprehensive study, we also carry out an error analysis:  we compute the truncation and phase errors of several finite difference schemes. 
We finally investigate the convergence of these schemes.  \\

The remaining of this chapter is organized as follows: in the next section, we present finite difference schemes with some discrete version for the fractional derivative and we study the continuous amplification factor of the linearized Fowler model. Sections \ref{stability} and \ref{erreur} are, respectively, devoted to the stability and error analysis. 
The paper ends with some remarks in section \ref{conclusion4}. 

\section{Preliminaries \label{preliminaire}}

\subsection{Finite difference approximations}

The spatial discretization is given by a set of points ${x_{j}; j=0,...,N}$ and the discretization in time is represented by a sequence of times $t^0=0<...<t^n<...<T$. For the sake of simplicity we will assume constant step sizes $\delta x$ and $\delta t$ in space and time respectively. The discrete solution at a point will be represented by $u^n_j \approx u(t^n,x_j)$. 
The schemes consist in computing approximate values $u^n_j$ of solution to \eqref{fowlereqn4} on $[n\delta t, (n+1) \delta t[\times [j \delta x, (j+1) \delta x[$ for $n \in \mathbb{N}$ and $j \in \mathbb{N}$ thanks to the following relation:
\begin{equation}
\label{FDscheme}
\frac{u_{j}^{n+1}-u_{j}^{n}}{\delta t}+ F(u^{n}_{j-1}, u^n_{j}, u^n_{j+1})- \epsilon \, \frac{u_{j+1}^{n}-2u_{j}^{n}+u_{j-1}^{n}}{\delta x^{2}}+ \eta \, \I_{\delta x}[u^{n}]_{j}=0,
\end{equation} 
where $\I_{\delta x}$ and $F$ are, respectively, the discretizations of the nonlocal and nonlinear terms. Note that the Laplacian term is discretized using centred finite difference approximation.
We begin by considering two discretizations $\I_{\delta x}^{1}, \I_{\delta x}^2$ for the operator $\I$ corresponding to formulae \eqref{nonlocalterm3b}  and  \eqref{intformula}, respectively. In both cases, we use a basic quadrature rule on the mesh $\left([j \delta x, (j+1)\delta x [  \right)_{j \in \mathbb{N}}$ to approximate each integral and we use a finite difference approximation of the derivative:
\begin{equation} \label{discretization14}
\I^{1}_{\delta x}[\varphi]_{j}=\delta x^{-4/3} \sum^{+ \infty}_{l=1} l^{-1/3} \left(\varphi_{j-l+1}-2\varphi_{j-l} + \varphi_{j-l-1}\right),
\end{equation}
\begin{equation}\label{discretization24}
\I^{2}_{\delta x}[\varphi]_{j} = \frac{4}{9} \delta x^{-4/3} \sum^{+ \infty}_{l=1} l^{-7/3} \left(\varphi_{j-l}-\varphi_{j}+ l \, \left( \frac{\varphi_{j+1}-\varphi_{j-1}}{2}\right)   \right).
\end{equation}
Let us remark that we begin the sums at $l=1$ in order to avoid the singularity of $1/|z|^{1/3}$ and $1/|z|^{7/3}$ at $z = 0$. We will comment later on the truncation of the series, see Section \ref{erreur}. Let us simply note that if $\varphi_j = 0$ for all $j<0$ then the series \eqref{discretization14} is in fact a finite sum.
Since the spatial mesh is given by $\left([j \delta x, (j+1)\delta x [  \right)_{j \in \mathbb{N}}$, we will indeed assume that $\varphi_j = 0$ for all $j<0$. 

\begin{remark}
Using fractional calculus, we could also consider, for any causal function $\varphi$, the standard Gr\"unwald-Letnikov formula for the fractional derivative $\I$. Indeed, using the expression \eqref{riemann}, $\I$ can be approximated by the following two formulae
\begin{eqnarray}
\I^{3}_{\delta x}[\varphi]_j & = & \frac{\Gamma(2/3) }{\delta x^{4/3} } \sum_{l \geq 0} (-1)^l \begin{pmatrix} 4/3 \\
                         l \\
\end{pmatrix} \varphi_{j-l} = \frac{\Gamma(2/3) }{\delta x^{4/3} } \sum_{l \geq 0} \frac{\Gamma(l-4/3)}{\Gamma(l+1)\Gamma(-4/3)} \varphi_{j-l}, 
\label{grunwald}
\end{eqnarray}
and
\begin{eqnarray}
\I^{3}_{\delta x}[\varphi]_j & = & \frac{\Gamma(2/3) }{\delta x^{4/3}} \sum_{l \geq 0}  \begin{bmatrix} -2/3 \\
                         l \\
\end{bmatrix} ( \varphi_{j-l+1}-2\varphi_{j-l}+ \varphi_{j-l-1}),
\end{eqnarray}
where, for all $\alpha >0$ and $k\in \mathbb{N}$  we denote by $\displaystyle {\alpha \choose k}$ the binomial coefficient defined by
$$ \displaystyle {\alpha \choose k}:= \frac{\alpha(\alpha-1) \dots (\alpha - k +1)}{k!} = (-1)^k \frac{\Gamma(k-\alpha)}{\Gamma(-\alpha) \Gamma(k+1)}$$ 
and $\begin{bmatrix} p \\
                         k \\
\end{bmatrix}$ denotes the negative binomial given by
\begin{equation*}
\begin{bmatrix} p \\
                         k \\
\end{bmatrix} = \frac{p(p+1) \cdots  (p+k-1)}{k!}= (-1)^k { -p \choose k}. 
\end{equation*}
For more details about Gr\"unwald-Letnikov derivatives, we refer the reader to the book \cite{pod}.
\end{remark} 

To analyze the stability of the discrete problem \eqref{FDscheme} using Fourier analysis, we investigate the following linearized explicit scheme
\begin{eqnarray}
\label{FDlinear}
\frac{u_{j}^{n+1}-u_{j}^{n}}{\delta t} + v \, \frac{ u^n_{j}-u^n_{j-1}}{\delta x}    -  \epsilon \, \frac{u^{n}_{j+1}-2u^{n}_{j}+u^{n}_{j-1} }{\delta x^{2}}  + \eta \, \I_{\delta x}[u^{n}]_{j}=0, 
\end{eqnarray} 
where $v$ is a positive constant. 
\begin{remark}
In the case where we consider that $v$ is a non-positive constant, $\partial_x u$ is discretized using a downstream finite difference approximation and so $F$ is given by
\begin{equation*}
F(u^n_{j-1},u^n_{j}, u^n_{j+1}) = v \frac{u^n_{j+1}-u^n_j }{\delta x}.
\end{equation*}
\end{remark}

Therefore, taking into account the discretization \eqref{discretization14}, the numerical scheme is written as follows:
\begin{eqnarray}
u^{n+1}_j&=& \frac{\epsilon \, \delta t}{\delta x^2}  u^n_{j+1} +  \left(1-\frac{v \delta t}{\delta x}-2\frac{\epsilon \, \delta t}{\delta x^2} \right) u^n_j  +\left( \frac{v \, \delta t}{\delta x} + \frac{\epsilon  \, \delta t}{\delta x^2} \right) u^n_{j-1} \nonumber \\
&-& \frac{\eta \, \delta t}{\delta x^{4/3}} \sum^{+ \infty}_{l=1} l^{-1/3} \left(u^n_{j-l+1}-2 u^n_{j-l} + u^n_{j-l-1}\right),
\label{schema1}
\end{eqnarray}
and since
\begin{eqnarray*}
\sum^{+ \infty}_{l=1} l^{-1/3} \left(u^n_{j-l+1}-2 u^n_{j-l} + u^n_{j-l-1}\right) &=& \sum_{l=2}^{+ \infty}
\left[(l+1)^{-1/3}-2l^{-1/3}+(l-1)^{-1/3}\right] u^{n}_{j-l} \\ 
&-& \, u^n_j - (2-2^{-1/3})u^n_{j-1},
\end{eqnarray*}
the numerical scheme \eqref{schema1} reads
\begin{eqnarray}
u^{n+1}_j&=&  \frac{\epsilon \, \delta t}{\delta x^2} \, u^{n}_{j+1}+ \left(1-\frac{v \, \delta t}{\delta x}-2\frac{\epsilon \, \delta t}{\delta x^2}-\frac{ \eta \, \delta t}{\delta x^{4/3}}\right) \, u^{n}_{j}+ \left(\frac{v \, \delta t}{\delta x}+ \frac{\epsilon \, \delta t}{\delta x^2} + (2-2^{-1/3})\frac{\eta  \, \delta t}{\delta x^{4/3}}\right) \, u^{n}_{j-1} \nonumber \\
&-&   \frac{\eta \, \delta t}{\delta x^{4/3}} \sum_{l=2}^{+ \infty}
\left[(l+1)^{-1/3}-2l^{-1/3}+(l-1)^{-1/3}\right] u^{n}_{j-l}.
\label{schema1b}
\end{eqnarray}

Considering now the discretization \eqref{discretization24}, the numerical scheme \eqref{FDlinear} can be written as follows:
\begin{eqnarray*}
u^{n+1}_j & = &  \frac{\epsilon \, \delta t}{\delta x^2}u^{n}_{j+1} + (1- \frac{v  \, \delta t}{\delta x}- 2  \frac{\epsilon \, \delta t}{\delta x^2})u^{n}_j + ( \frac{v \, \delta t}{\delta x} + \frac{\epsilon \, \delta t}{\delta x^2}) u^n_{j-1}\\ 
&-& \frac{4}{9} \frac{\eta \, \delta t}{\delta x^{4/3}}
 \sum^{+ \infty}_{l=1} l^{-7/3} \left(u^n_{j-l}-u^n_{j}+  l \left( \frac{u^n_{j+1}-u^n_{j-1}}{2} \right)   \right).
\end{eqnarray*}
Recall that the Riemann zeta function, for $\mbox{Re}(s)>1$
 $$\zeta(s) = \sum_{n =1}^{\infty} n^{-s}. $$ 
 Since
\begin{eqnarray*}
 \sum^{+ \infty}_{l=1} l^{-7/3} \left(u^n_{j-l}-u^n_{j}+ l \left( \frac{u^n_{j+1}-u^n_{j-1}}{2} \right) \right) &=&  \frac{1}{2} \zeta(\frac{4}{3}) \, u^n_{j+1} - \zeta(\frac{7}{3}) \, u^n_j \\ 
 &-& \left( \frac{1}{2} \zeta(\frac{4}{3}) -  1  \right)u^n_{j-1} + \sum_{l=2}^{+\infty} l^{-7/3} \, u^n_{j-l}, 
\end{eqnarray*}
with $\zeta(\frac{4}{3}) \approx 3.601$, $\zeta(\frac{7}{3}) \approx 1.415 $, 
the numerical scheme reads
\begin{eqnarray}
u^{n+1}_j&=& \left(  \frac{ \epsilon \, \delta t}{\delta x^2}-  \frac{4}{9} \frac{\eta \, \delta t}{\delta x^{4/3}} \frac{1}{2} \zeta(\frac{4}{3})  \right) u^{n}_{j+1}+ \left(1-\frac{v \, \delta t}{\delta x}- 2 \frac{\epsilon \, \delta t}{\delta x^2} +  \frac{4}{9} \frac{\eta \,  \delta t  }{\delta x^{4/3}} \zeta(\frac{7}{3})  \right)u^n_j \nonumber \\ 
&+& \left(\frac{v \, \delta t}{\delta x}+ \frac{\epsilon \, \delta t}{\delta x^2} + \frac{4}{9} \frac{\eta \, \delta t  }{\delta x^{4/3}}(\frac{1}{2} \zeta(\frac{4}{3})-1)\right) \, u^{n}_{j-1}-  \frac{4}{9} \frac{\eta \,\delta t}{\delta x^{4/3}} \sum_{l=2}^{+\infty} l^{-7/3} \, u^{n}_{j-l}. \label{chepas0}
\end{eqnarray}

\begin{remark}
If the Fowler equation \eqref{fowlereqn4} satisfied the maximum principle, a classical way to get sufficient conditions for the $L^\infty$-stability of the scheme would be to ensure that $u^{n+1}_j$ is a convex
combination of $(u^n_j)_{j \in \mathbb{N}}$. Though one can easily check that all coefficients sum up to 1, we remark that $(l+1)^{-1/3}-2l^{-1/3}+(l-1)^{-1/3}>0 $ because the function $x \rightarrow x^{-1/3}$ is convex and $-\frac{4}{9}   \frac{\eta \, \delta t}{\delta x^{4/3}} l^{-7/3} <0$  for all $l>1$. Thus, $u^{n+1}_j$ is not a convex
combination of $(u^n_j)_{j \in \mathbb{N}}$. 
To get conditions of numerical stability we have to rely on the Von Neumann method. 
\end{remark}

\subsection{The continuous amplification factor}

In this section, we are going to study the amplification factor of the following equation
\begin{equation}
\partial_{t} u(t,x) + v \, \partial_{x} u(t,x) -\epsilon \, \partial^2_{xx} u(t,x) + \eta \, \I[u(t, \cdot)](x) = 0.
\label{linearized}
\end{equation}   
Then, $u(t,x) = e^{i k x + \sigma t}$ is a solution to \eqref{linearized} if and only if the following dispersion relation is satisfied
\begin{equation*}
\sigma+ i vk+\epsilon k^2- \eta |k|^{4/3}\frac{1}{2} \Gamma(\frac{2}{3}) \left( 1-i \sqrt{3} \, sign(k) \right) =0,
\end{equation*}
where $k \in \R$ and $\sigma \in \mathbb{C}$. Indeed, we have 
$u_t(t,x) = \sigma u(t,x), u_x(t,x) =  iku(t,x), u_{xx}(t,x) =  -k^2u(t,x)$
and
\begin{eqnarray*}
\I[u(t, \cdot)](x) &=&  \int_{0}^{+\infty} |\xi|^{-1/3} (-k^2) e^{ik(x-\xi)+\sigma t} \, d\xi, \\
&=& - k^2 u(t,x) \int_{0}^{+\infty} |\xi|^{-1/3} e^{-i k \xi} \, d\xi, \\
&=& - k^2 u(t,x) \left[ \int_{0}^{+\infty} |\xi|^{-1/3} \cos(k \xi) \, d\xi + i \int_{0}^{+\infty} |\xi|^{-1/3} \sin(k \xi) \, d\xi \right], \\
&=&  \left[ - |k|^{4/3} \frac{1}{2}\Gamma(\frac{2}{3}) + k |k|^{1/3} \frac{\sqrt{3}}{2} \Gamma(\frac{2}{3})\right]  u(t,x),
\end{eqnarray*}
where we have used Fresnel integrals. \\
Hence
the multiplicative factor which enables to get the solution at the time $t_{n+1}$ from the solution at the time $t_n$ is
\begin{equation}
 G_{cont}(k) = e^{-\delta t \, \phi(k)},
 \label{exactfactor}
\end{equation} 
where $\phi(k)= \epsilon \, k^{2} - \eta \frac{1}{2} \Gamma(\frac{2}{3}) \, \vert k \vert^{4/3} + i  \left( \eta \frac{\sqrt{3}}{2} \Gamma(\frac{2}{3}) \,  k \vert k \vert^{1/3} +  \, v  k \right) $. Therefore  
\begin{equation*}
| G_{cont}(k)| = e^{-\delta t    \left( \epsilon \, k^{2} - \eta \frac{1}{2} \Gamma(\frac{2}{3}) \, \vert k \vert^{4/3} \right) }.
\end{equation*}

\begin{figure}[h!]
	\centering
	\includegraphics[scale=0.85]{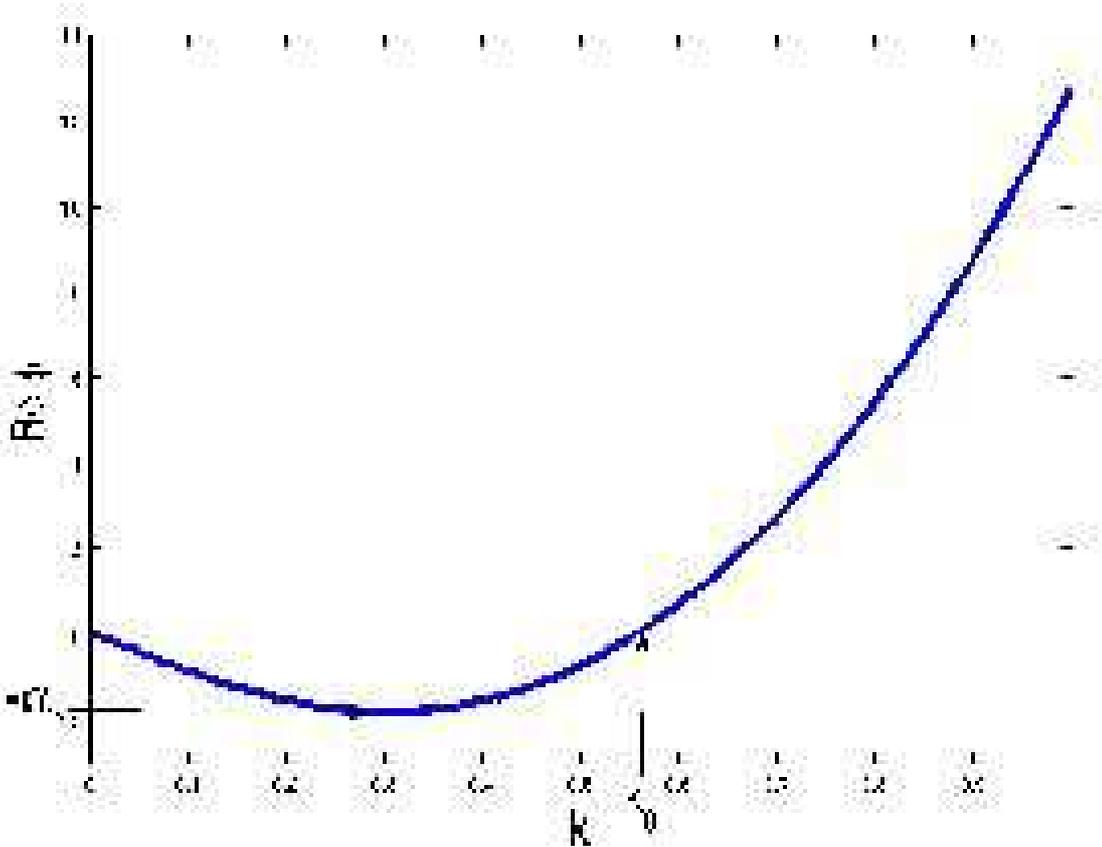} 
	\caption{Behaviour of $ \mbox{Re}\left(  \phi \right)   $ for $\eta, \epsilon$ fixed. $k_0 = \left( \frac{1}{2}  \Gamma\left(\frac{2}{3} \right) \frac{\eta}{\epsilon} \right)^{3/2} $ is the threshold frequency.
}
\label{phiencore}
\end{figure} 
Figure \ref{phiencore} shows that the modulus of the continuous amplification factor during one time step is controlled by  $e^{\alpha_* \delta t}$, with $\alpha_* := - \min \mbox{Re} \left( \phi\right)  = -\mbox{ Re } \phi(k^*)=  \frac{4}{27} \left( \frac{1}{2} \Gamma\left( \frac{2}{3}\right) \right)^{3} \frac{\eta^3}{\epsilon^2}$, where 
\begin{equation*}
k_* = \left(\frac{1}{3} \Gamma\left( \frac{2}{3}\right) \frac{\eta}{ \epsilon} \right)^{3/2}.
\end{equation*} 
Thereby, the exact continuous amplification is maximum for frequency $k_*$, and its modulus is bigger than 1 only for frequencies in the range $(0, k_0]$. 
The magnitude of this amplification during one time step will also be proportional to   $\delta t$. Obviously, this phenomenon affects only the low frequencies in the range $(0, k_0]$ and strongly depends on the choice of parameters $\eta$ and $\epsilon.$ This is why the standard definitions of stability are not adapted for this model because  they do not take into account the possibility of amplification of certain frequencies.  \\
And since high frequencies are usually
responsible of numerical instabilities, we are going to focus our attention on the high frequencies which are quickly dampened in Fowler's continuous model in order to exhibit numerical stability conditions. 



\section{Stability analysis \label{stability}}

The purpose of this section is to study the numerical stability of schemes introduced in the previous section and to exhibit stability criteria. We recall that 
the numerical stability enables to ensure that the difference between the approximated solution and the exact solution remains bounded for all $T >0$ with $\delta x, \delta t$ given. To get numerical stability criteria, we consider the Von Neumann or Fourier method. In this approach, we assume that the discrete solution is written in as a single Fourier mode
\begin{equation} \label{mode}
u_{j}^{n}=\hat{u}^{n}_{k} e^{ikx_{j}},
\end{equation} 
where $k \in \mathbb{Z}$ is the wave number.
Injecting \eqref{mode} in the numerical scheme \eqref{FDlinear}, we get
\begin{equation}  \label{fourierdiscret}
\hat{u}^{n+1}_{k}=g(\delta x, \delta t, k) \hat{u}^{n}_{k},
\end{equation}
where $g$ is the discrete amplification factor. 
In what follows, for simplicity, we denote indifferently
$g(\delta x, \delta t, k)= g(\delta x, \delta t, \theta)$, where $\theta = k \delta x$.\\

\begin{remark}
Note that due to the aliasing phenomenon it is enough to study the discrete amplification factor for $\theta \in [0,\pi]$.\\
\end{remark}

Following the previous discussion concerning the notion of numerical stability (see Section \ref{preliminaire}), we introduce the following definition:

\bigskip 
 
\begin{definition}
\label{defstab}
We say that a numerical scheme which approximates the linearized Fowler equation problem is stable if the high frequencies are strongly stable that is to say:
$$
\exists \, 0<\theta_0 < \pi \mbox{ such that } \forall \theta \in (\theta_0, \pi], |g(\delta x, \delta t, \theta)|<1,
$$
where $g$ is the discrete amplification factor.
\label{definitionstability} 
\end{definition}

\bigskip 

\begin{lemme} \label{cercleremark} Let $a,b \in \R$ and $d \in \R^+$. Then we have 
$$\forall \theta \in [0, 2 \pi], |a + b e^{- i\theta}| \leq d \hspace{0.2 cm} \mbox{ if and only if } \hspace{0.2 cm}  a + |b| \leq d \mbox { and } a - |b| \geq -d.$$   
\end{lemme}

\begin{proof}
We can easily check this property, see Figure \ref{cercle}. \\
\begin{figure}[h!]
	\centering
	\includegraphics[scale=0.65]{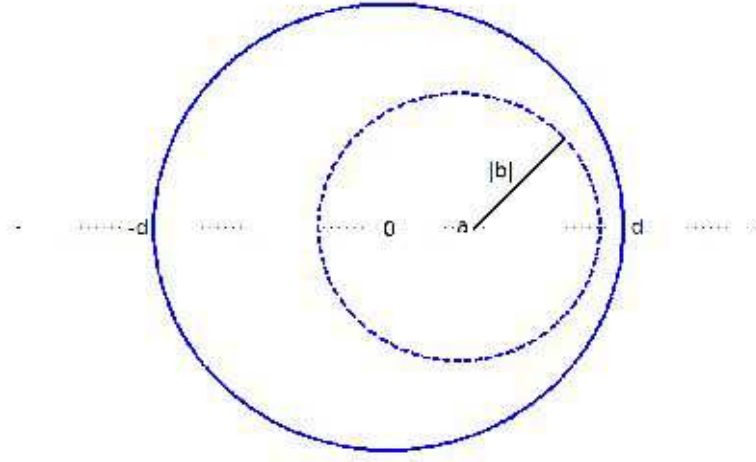} 
	\caption{Dashed circle (resp. continuous circle) is centred at $a$ (resp. 0) and of radius $|b|$ (resp. $d$).}
\label{cercle}
\end{figure} 
\end{proof}

\begin{proposition}
The finite difference scheme \eqref{FDlinear} is stable in the sense of Definition \ref{definitionstability} if $\delta x$ and $ \delta t $ satisfy the following conditions:
\begin{itemize}
\item For $\I_{\delta x}^1$: 
\begin{equation}
 \frac{ v \, \delta t}{\delta x}+ 2 \frac{ \epsilon \, \delta t}{\delta x^2}+ (2-2^{-1/3})  \frac{\eta \, \delta t}{\delta x^{4/3}} \leq 1,
 \label{conditionstab01}
\end{equation}
\item  For $ \I_{\delta x}^{2} $: 
\begin{equation}
\frac{v\, \delta t}{\delta x}+ 2 \frac{\epsilon \, \delta t}{\delta x^2} +  \frac{4}{9} \left(\zeta(\frac{4}{3}) -1 \right) \frac{\eta \, \delta t}{\delta x^{4/3}}  \leq 1,
\label{conditionstab02}
\end{equation}
\end{itemize}
and if moreover, the space-step $\delta x$ is small enough in order that 
\begin{itemize}
\item 
For $\I_{\delta x}^1$: 
\begin{equation}
(1-2^{-1/3}) \frac{\eta \, \delta t}{\delta x^{4/3}} \leq  2\frac{ \epsilon \, \delta t}{\delta x^2} \sin^{2}(\frac{\theta_0}{2}),
\label{condbiz1}
\end{equation}

\item  
For $ \I_{\delta x}^{2} $: 
\begin{equation}
 \frac{4}{9} \left( \zeta(\frac{7}{3})-1 + \zeta(\frac{4}{3}) \right) \frac{\eta \, \delta t}{\delta x^{4/3}}  \leq 2  \frac{\epsilon \, \delta t}{\delta x^2} \sin^{2}(\frac{\theta_0}{2}),
 \label{condbiz2}
 \end{equation}

\end{itemize}
where $\theta_0$ designates the stability threshold frequency. 
\end{proposition}

\begin{proof}
\noindent \textbf{For $\I_{\delta x}^1$.} 

For the numerical scheme \eqref{schema1b},
the amplification factor is given by:
\begin{eqnarray}
g_1(\delta x, \delta t, \theta)&=&1-\frac{v \, \delta t}{\delta x}-2\frac{ \epsilon \, \delta t}{\delta x^2}(1-\cos \theta)-  \frac{\eta \, \delta t}{\delta x^{4/3}}+ \left( \frac{v \, \delta t}{\delta x}+ (2-2^{-1/3})\frac{\eta  \, \delta t}{\delta x^{4/3}}\right) e^{-i \theta} \nonumber \\
&-& \frac{\eta \, \delta t}{\delta x^{4/3}} \sum_{l=2}^{\infty}
\left[(l+1)^{-1/3}-2l^{-1/3}+(l-1)^{-1/3}\right] e^{-i l \theta},
\label{amplifidiscret1}
\end{eqnarray}
where $\theta= k\delta x$. 
Since, for all $N \in \mathbb{N}$
\begin{equation*}
\sum_{l=2}^{N} \left[ 
(l+1)^{-1/3}-2l^{-1/3}+(l-1)^{-1/3} \right] =(N+1)^{-1/3}-N^{-1/3} - 2^{-1/3} + 1,
\end{equation*}
then
\begin{equation}
\sum_{l=2}^{+\infty}\left[ 
(l+1)^{-1/3}-2l^{-1/3}+(l-1)^{-1/3} \right] =1-2^{-1/3} > 0.
\label{somme}
\end{equation}

Thus, from \eqref{amplifidiscret1}, to have 
$|g_1(\delta x, \delta t, \theta)| \leq 1 $ it is sufficient to have
\begin{eqnarray}
 \left| 1-\frac{v \, \delta t}{\delta x} - 4  \sin^{2}(\frac{\theta}{2})\frac{\epsilon \, \delta t}{\delta x^2} - \frac{\eta \, \delta t}{\delta x^{4/3}}+ \left( \frac{v \, \delta t}{\delta x} + (2-2^{-1/3})\frac{ \eta  \, \delta t}{\delta x^{4/3}}\right)e^{-i \theta}  \right|& \leq& \hspace{1 cm}\nonumber \\
  1  -\frac{\eta  \, \delta t}{\delta x^{4/3}}(1-2^{-1/3}),
\label{pasinspire}
\end{eqnarray}
where we assume that 
\begin{equation}
\frac{\eta \, \delta t}{\delta x^{4/3}}(1-2^{-1/3})<1.
\label{hypothese1}
\end{equation}
Next from Lemma \ref{cercleremark}, \eqref{pasinspire} is satisfied if and only if we have
\begin{equation*}
  \left\{
      \begin{aligned}
       1-4\frac{ \epsilon \, \delta t}{\delta x^2} \sin^{2}(\frac{\theta}{2})+ \frac{\eta \, \delta t}{\delta x^{4/3}}(1-2^{-1/3}) \leq 1  -(1-2^{-1/3}) \frac{\eta \, \delta t}{\delta x^{4/3}}, \\
1-2 \frac{v \,  \delta t}{\delta x}-4\frac{\epsilon \, \delta t}{\delta x^2} \sin^{2}(\frac{\theta}{2})-(3-2^{-1/3}) \frac{\eta \, \delta t}{\delta x^{4/3}} \geq -\left(1 -(1-2^{-1/3}) \frac{\eta \, \delta t}{\delta x^{4/3}}\right).
      \end{aligned}
    \right.
\end{equation*}
A sufficient condition is then
\begin{eqnarray}
(1-2^{-1/3}) \frac{\eta \, \delta t}{\delta x^{4/3}} \leq  2\frac{ \epsilon \, \delta t}{\delta x^2} \sin^{2}(\frac{\theta}{2}), \label{constab0}\\
\frac{v \, \delta t}{\delta x} + 2 \frac{ \epsilon \, \delta t}{\delta x^2}+(2-2^{-1/3})  \frac{\eta \, \delta t}{\delta x^{4/3}} \leq 1. \nonumber \label{constab0b}
\end{eqnarray}
Let $0<\theta_0<\pi$. Then, for all $\theta \in (\theta_0, \pi] $, condition \eqref{constab0} can be rewritten as
\begin{eqnarray*}
(1-2^{-1/3}) \frac{\eta \, \delta t}{\delta x^{4/3}} \leq  2\frac{ \epsilon \, \delta t}{\delta x^2} \sin^{2}(\frac{\theta_0}{2}). 
\end{eqnarray*}
Therefore, the numerical scheme \eqref{FDlinear}  with the discretization $\I_{\delta x}^1$ is stable in the sense of Definition \ref{definitionstability} if the space and time steps $\delta t, \delta x$ satisfy the following conditions
\begin{eqnarray}
\delta x^{2/3} \leq \frac{2}{(1-2^{-1/3})} \sin^{2}(\frac{\theta_0}{2}) \frac{\epsilon}{\eta}, \label{constab00b}\\
\frac{v \, \delta t}{\delta x} + 2 \frac{ \epsilon \, \delta t}{\delta x^2}+(2-2^{-1/3})  \frac{\eta \, \delta t}{\delta x^{4/3}} \leq 1. \label{constab0b}
\end{eqnarray}
Note that from condition \eqref{constab0b}, we can see that hypothesis \eqref{hypothese1} is satisfied.  \\

\noindent \textbf{For $\I_{\delta x}^2$.}
Injecting \eqref{mode} in \eqref{chepas0}, the amplification factor $g_2$ associated to this scheme is given by:
\begin{eqnarray}
g_2( \delta x,\delta t,  \theta)&:=&  e^{i \theta} \left( \frac{ \epsilon \, \delta t}{\delta x^2} - \frac{1}{2} \zeta(\frac{4}{3}) \frac{4}{9} \frac{ \eta \, \delta t}{\delta x^{4/3}}  \right) + 1 - \frac{v \, \delta t}{\delta x} - 2 \frac{\epsilon \, \delta t}{\delta x^2} +  \frac{4}{9}  \zeta(\frac{7}{3}) \frac{\eta \, \delta t}{\delta x^{4/3}} \nonumber \\ 
&+& e^{-i \theta}\left( \frac{v \, \delta t}{\delta x} +    \frac{ \epsilon \, \delta t}{\delta x^2} +    \frac{4}{9} (\frac{1}{2} \zeta(\frac{4}{3})-1 ) \frac{\eta \, \delta t}{\delta x^{4/3}}  \right) - \frac{4}{9}\frac{ \eta \, \delta t}{\delta x^{4/3}} \sum_{l \geq 2} l^{-7/3} e^{-i \theta l}, \nonumber
 \\ 
&=& 1-\frac{v \, \delta t}{\delta x}- 4 \frac{\epsilon \, \delta t}{\delta x^2} \sin^{2}(\frac{\theta}{2}) + \frac{4}{9}  \zeta(\frac{7}{3}) \, \frac{\eta \, \delta t}{\delta x^{4/3}} -i \,  \frac{4}{9}\zeta(\frac{4}{3}) \, \frac{\eta \, \delta t}{\delta x^{4/3}}  \sin \theta \nonumber \\
&+& \left(\frac{v \,\delta t}{\delta x}-  \frac{4}{9} \frac{\eta \, \delta t}{\delta x^{4/3}}\right) e^{-i \theta}- \frac{4}{9} \frac{\eta \, \delta t  }{\delta x^{4/3}}\sum_{l=2}^{+\infty} l^{-7/3} e^{-i \theta l},  \nonumber \\
&=& 1-\frac{v \, \delta t}{\delta x}- 4 \frac{\epsilon \, \delta t}{\delta x^2} \sin^{2}(\frac{\theta}{2}) + \frac{4}{9} \left( \zeta(\frac{7}{3})  -  \zeta(\frac{4}{3}) \cos \theta\right)  \, \frac{\eta \, \delta t}{\delta x^{4/3}}   \nonumber \\
&+&  \left(\frac{v \,\delta t}{\delta x} +  \frac{4}{9} \left( \zeta ( \frac{4}{3}) -1\right) \frac{\eta \, \delta t}{\delta x^{4/3}}\right)e^{-i \theta}- \frac{4}{9} \frac{\eta \, \delta t  }{\delta x^{4/3}}\sum_{l=2}^{+\infty} l^{-7/3} e^{-i \theta l}.     
\label{amplifidiscret2}      
\end{eqnarray}
Since 
$$\sum_{l=2}^{+\infty} l^{-7/3} = \zeta(\frac{7}{3})-1 \approx 0.415,$$
from \eqref{amplifidiscret2}, $|g_2(\delta x, \delta t, \theta)|\leq 1 $ if
\begin{eqnarray}
\left| 1-\frac{v \, \delta t}{\delta x}- 4 \frac{\epsilon \, \delta t}{\delta x^2} \sin^{2}(\frac{\theta}{2}) + \frac{4}{9} \left( \zeta(\frac{7}{3})  -  \zeta(\frac{4}{3}) \cos \theta\right)  \, \frac{\eta \, \delta t}{\delta x^{4/3}} +\left(\frac{v \,\delta t}{\delta x} +  \frac{4}{9} \left( \zeta ( \frac{4}{3}) -1\right) \frac{\eta \, \delta t}{\delta x^{4/3}}\right)e^{-i \theta} \right| &\leq& \nonumber \\ 
\hspace{-4 cm} 1 - \frac{4}{9} \left( \zeta(\frac{7}{3}) - 1\right) \frac{\eta \, \delta t  }{\delta x^{4/3}},
 \label{pasinspire2}
\end{eqnarray}
where we assume that
\begin{equation}
\frac{4}{9} \left( \zeta(\frac{7}{3}) - 1\right) \frac{\eta \, \delta t  }{\delta x^{4/3}} < 1.
\label{hypothese2}
\end{equation}

From Lemma \ref{cercleremark}, we have that \eqref{pasinspire2} is satisfied if and only if
\begin{equation*}
  \left\{
      \begin{aligned}
       1-4 \frac{\epsilon \, \delta t}{\delta x^2} \sin^{2}(\frac{\theta}{2}) + \frac{4}{9} \left( \zeta(\frac{7}{3}) -1 + \zeta(\frac{4}{3})(1- \cos \theta ) \right)\frac{\eta \, \delta t}{\delta x^{4/3} } \leq 1 - \frac{4}{9}  \left(\zeta(\frac{7}{3}) -1 \right)\frac{\eta \, \delta t}{\delta x^{4/3} },  \\
       1 - 2 \frac{v \, \delta t}{\delta x} - 4 \frac{\epsilon \, \delta t}{\delta x^2} \sin^{2}(\frac{\theta}{2}) + \frac{4}{9}  \left(\zeta(\frac{7}{3}) + 1 - \zeta(\frac{4}{3}) (1 + \cos \theta) \right) \frac{\eta \, \delta t}{\delta x^{4/3} } \geq - 1 + \frac{4}{9}  \left(\zeta(\frac{7}{3}) - 1 \right) \frac{\eta \, \delta t}{\delta x^{4/3} }.
      \end{aligned}
    \right.
\end{equation*}
A sufficient condition is then
\begin{eqnarray}
    \frac{4}{9} \left( \zeta(\frac{7}{3})-1 + \zeta(\frac{4}{3}) \right)\frac{\eta \, \delta t}{\delta x^{4/3}} \leq 2  \frac{\epsilon \, \delta t}{\delta x^2} \sin^{2}(\frac{\theta}{2}), \label{condstab02} \\
\frac{v\, \delta t}{\delta x}+ 2 \frac{\epsilon \, \delta t}{\delta x^2} +  \frac{4}{9}  \left(\zeta(\frac{4}{3}) -1 \right) \frac{\eta \, \delta t}{\delta x^{4/3}} \leq 1.
\label{condstab02b}
\end{eqnarray}
Let $0<\theta_0<\pi$. Then, for all $\theta \in (\theta_0, \pi] $, condition \eqref{condstab02} is rewritten as
\begin{eqnarray}
\frac{4}{9} \left( \zeta(\frac{7}{3})-1 + \zeta(\frac{4}{3}) \right) \frac{\eta \, \delta t}{\delta x^{4/3}} \leq 2  \frac{\epsilon \, \delta t}{\delta x^2} \sin^{2}(\frac{\theta_0}{2}),
\label{condstab02bb} 
\end{eqnarray}
where $\zeta(\frac{7}{3})-1 + \zeta(\frac{4}{3}) \approx 4.02.$ \\
Note again that from condition \eqref{condstab02b}, we can see that hypothesis \eqref{hypothese2} is satisfied.

\end{proof}

\noindent \textbf{Notations.} We will denote by $CFL_{mod}^1$ and $CFL_{mod}^2$ the following modified Courant-Friedrichs-Lewy conditions 
\begin{equation*}
CFL_{mod}^1 = \frac{ v \, \delta t}{\delta x}+ 2\frac{ \epsilon \, \delta t}{\delta x^2}+ (2-2^{-1/3})  \frac{\eta \, \delta t}{\delta x^{4/3}} \leq 1, \\
\end{equation*} 

\begin{equation*}
CFL_{mod}^2 = \frac{v\, \delta t}{\delta x}+ 2 \frac{\epsilon \, \delta t}{\delta x^2} +  \frac{4}{9} \left(\zeta(\frac{4}{3}) -1 \right) \frac{\eta \, \delta t}{\delta x^{4/3}}  \leq 1. \\
\end{equation*} 

\bigskip

\noindent \textbf{Some remarks.}\\

\noindent \textbf{1.} Condition \eqref{conditionstab01} (resp. \eqref{conditionstab02}) can be seen as an extension                                                                   
of the classical CFL condition with in addition the anti-diffusive term $ \frac{\eta \, \delta t}{\delta x^{4/3}} $. This criterion is not
more restrictive than the usual condition of stability without the nonlocal operator which corresponds to the linearized
Burgers equation with viscous term. This condition is very restrictive on the space and time steps in particular because of the term $\frac{\epsilon \, \delta t}{\delta x^2}$ which stems from the explicit discretization of the Laplacian. In order to have less restrictive conditions, we can implicit some terms. 
For instance, if we decide to implicit the nonlocal and the Laplacian terms, condition \eqref{conditionstab01} (resp. \eqref{conditionstab02} ) is reduced to
\begin{equation*}
\frac{v \, \delta t}{\delta x}<1.
\end{equation*}
We find again the well-known CFL condition.\\

\noindent Figure \ref{amplifiAstab} shows the behaviour of amplification factors for $\I^1_{\delta x}$ and $\I^2_{\delta x}$. We can see, for  $\I^1_{\delta x}$, that the maximal value of $\delta t$ which ensures the numerical stability
is $\delta t_{\max} \approx 0.042$ and that for this value we have $CFL_{mod}^1 \approx 0.99$. 
\begin{figure}[h!]
	\centering
	\includegraphics[scale=0.7]{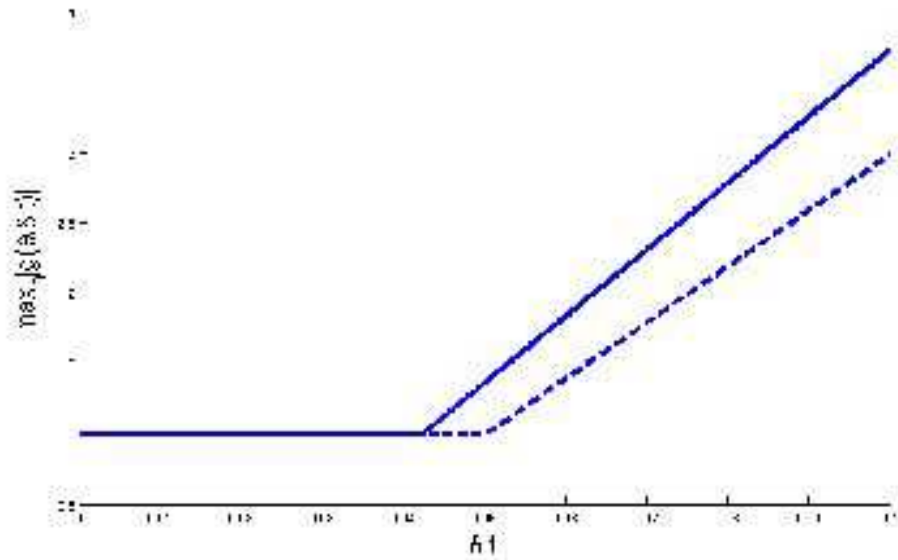} 
	\caption{Amplification factors for $\I^1_{\delta x}$ (blue line) and $\I^2_{\delta x}$ (dashed line).}
\label{amplifiAstab}
\end{figure}
Figure \ref{cnvscs} displays the behaviour of the modulus of the amplification factor with discretization $\I_{\delta x}^1$ as a function of $\theta$. We can notice that the high frequencies are strongly amplified. This phenomenon illustrates the numerical instability because high frequencies should be quickly dampened. \\
\begin{figure}[h!]
	\centering
	\includegraphics[scale=0.7]{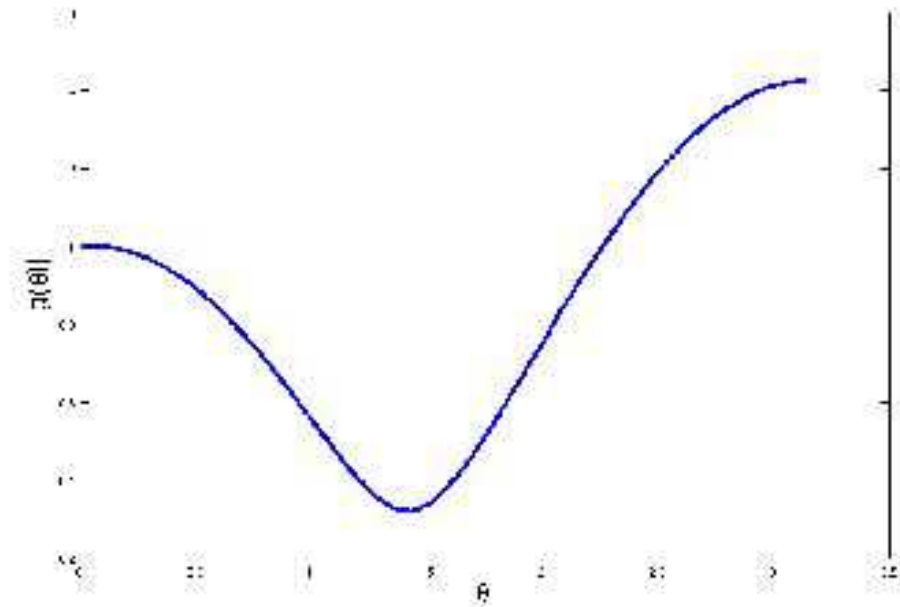} 
	\caption{Amplification factor for  $\I^1_{\delta x} $ with $CFL_{mod}^1 \approx 1.22$. }
\label{cnvscs}
\end{figure} 

Figure \ref{amplifbassefrequence} shows that the low frequencies are slowly amplified. This phenomenon is not due to the instability of numerical schemes but stems from the model. In Tables \ref{table13}, \ref{table23} and \ref{table33} (see Section \ref{erreur}), we have studied the quotient $\frac{|g_i|}{|G_{cont}|}$, $i=1,2$. We can see that globally the discrete schemes dampen more than the continuous problem when the stability conditions \eqref{conditionstab01} and \eqref{conditionstab02} are satisfied. \\

\begin{figure}[h!]
	\centering
	\includegraphics[scale=0.7]{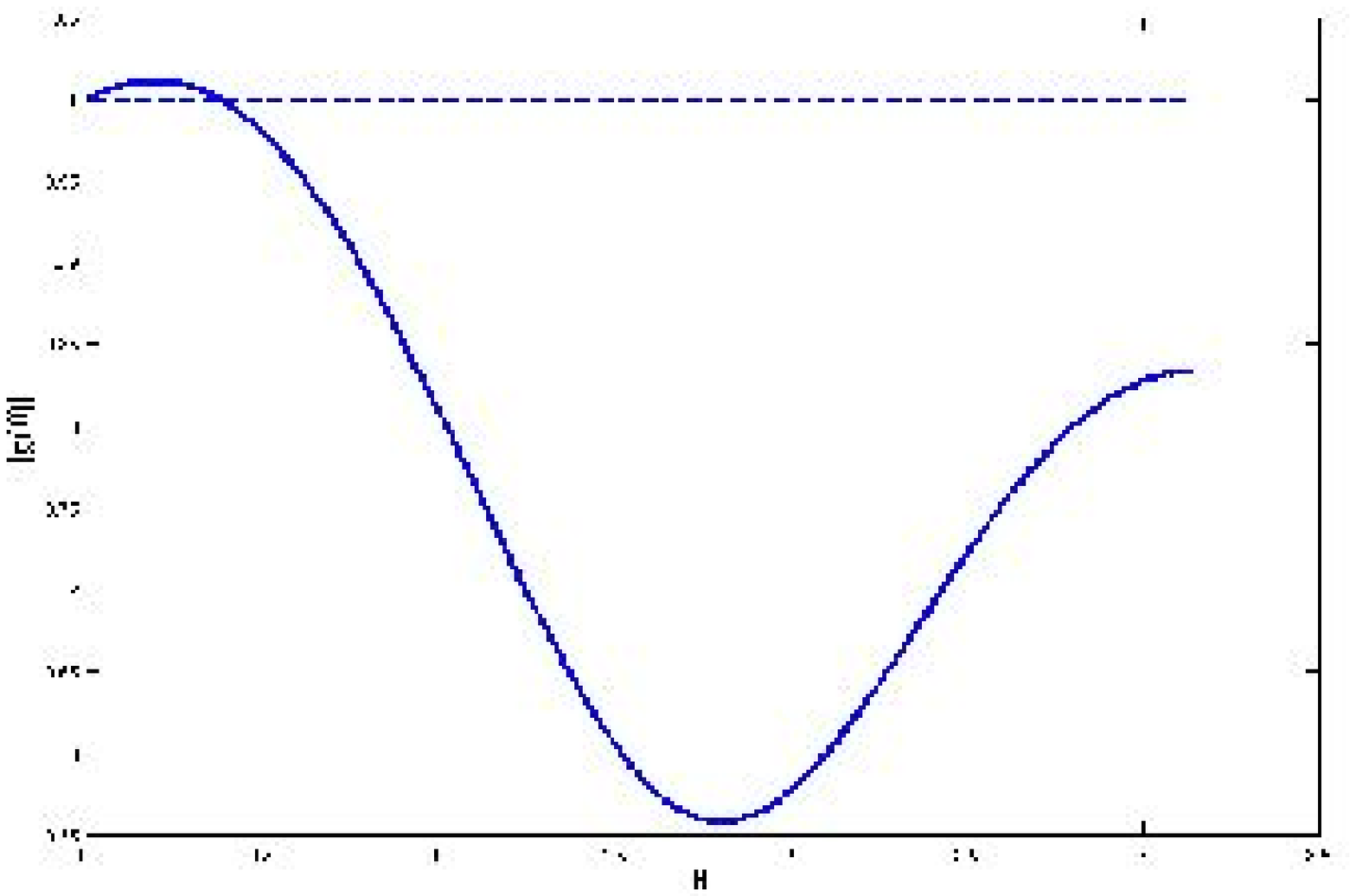} 
	\caption{Amplification factor $\I^1_{\delta x}$ with $\eta=8, v=1,\epsilon=0.5$ and $\delta x= 0.05, \delta t=0.001$. For these coefficients, $CFL_{mod}^1 \approx 0.94.$
\label{amplifbassefrequence}}
\end{figure}

\noindent \textbf{2.} Conditions \eqref{condbiz1} and \eqref{condbiz2}  are unusual and deserve some explanations.
The term proportional to $ \frac{\eta \, \delta t}{\delta x^{4/3}} $ represents the amount of nonlocal anti-diffusion  while the term proportional to $ \frac{\epsilon \, \delta t}{\delta x^2} $ corresponds to the amount of classical diffusion. Both conditions simply mean that, for frequencies  above  the threshold $\theta_0$, diffusion should control nonlocal anti-diffusion. \\
We can see that conditions \eqref{condbiz1} and \eqref{condbiz2}  cannot be satisfied for low frequencies. Indeed, for $\theta_0$ close to 0, these criteria impose to the space step to vanish, which is not possible. Let us note that this is coherent because the low frequencies are not ``strongly stable'', they are slowly amplified by the continuous problem. We can see in Figure \ref{stablemalgre} that condition  \eqref{condbiz1} is not necessary. Indeed, if we choose the threshold $\theta_0 = \pi/2$ ``large enough'', condition \eqref{condbiz1} reads 
\begin{equation}
\delta x \leq 0.25,
\label{cspasne}
\end{equation}
 and we have plotted $|g_1|$ in function of $\theta$ for $\delta x = 0.5 $ which does not satisfy the condition \eqref{cspasne} but we can still notice that the numerical scheme is stable. All numerical simulations that we performed confirm this statement. This leads us to think that condition  \eqref{condbiz1} (resp. \eqref{condbiz2}) is too pessimistic. In fact to estimate the  magnitude of sums
$$  \sum_{l=2}^{\infty}
\left[(l+1)^{-1/3}-2l^{-1/3}+(l-1)^{-1/3}\right] e^{-i l \theta}, $$
(resp.  $\sum_{l=2}^{+\infty} l^{-7/3} e^{-i l \theta } $), we just controlled the sum of the modulus
$$  \sum_{l=2}^{\infty} 
\left[(l+1)^{-1/3}-2l^{-1/3}+(l-1)^{-1/3}\right], $$
(resp. $\sum_{l=2}^{+\infty} l^{-7/3} $). In this manner, we probably miss some cancellation effect of the 
$e^{-i l \theta}$. But we could not find any other way to estimate these polylogarithm series. \\
  
\begin{figure}[h!]
	\centering
	\includegraphics[scale=0.7]{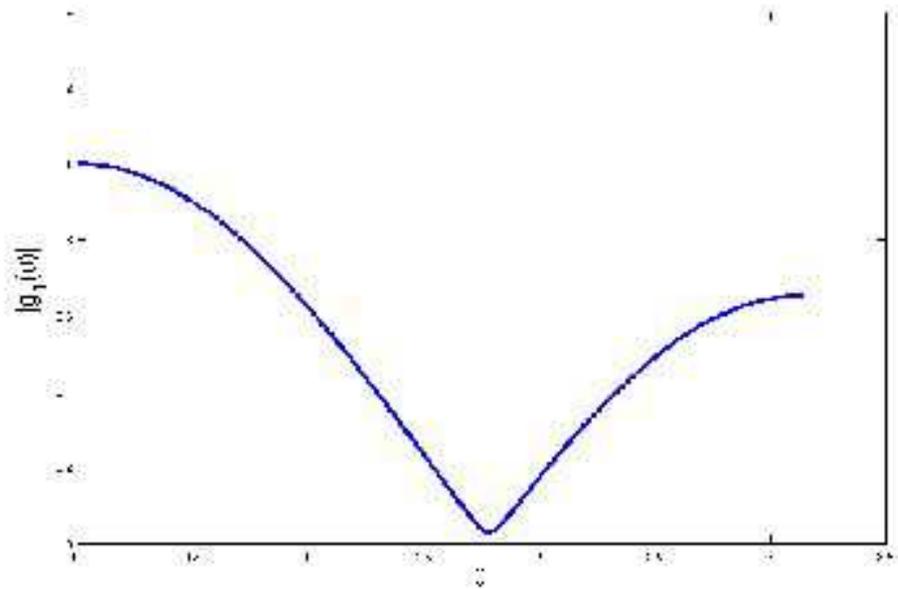} 
	\caption{Amplification factor $g_1$ for $\eta=v=1, \epsilon=0.1$ and $\delta x= 0.5, \delta t=0.01$. For these coefficients, we have $CFL_{mod}^1 \approx  0.0584$.}
\label{stablemalgre}
\end{figure}

\noindent \textbf{3.} Finally, in practice, the single condition \eqref{conditionstab01} (resp. \eqref{conditionstab02}) can be used to ensure the numerical stability of the scheme \eqref{FDlinear} with $\I^1_{\delta x}$ (resp.  $\I^2_{\delta x}$). We saw in Figures \ref{amplifbassefrequence} and \ref{stablemalgre} that the scheme with the discretization $\I^1_{\delta x}$ is stable if condition \eqref{conditionstab01} is satisfied. Figure \ref{factorfig} shows that the high frequencies are amplified, when condition \eqref{conditionstab02} is violated. This phenomenon is only due to numerical instability because the continuous problem quickly dampens the high frequencies. \\

\begin{figure}[h!]
	\centering
	\includegraphics[scale=0.7]{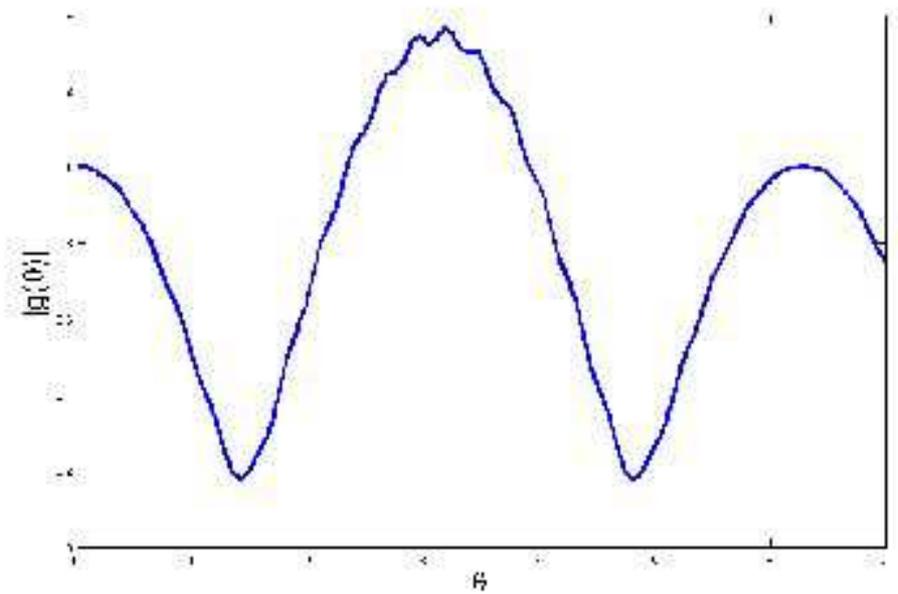} 
	\caption{Gain function for $\I_{\delta x}^2$ for  $\eta=v=\epsilon=1,$ $ \delta x=0.1, \delta t = 0.05$. For these coefficients, we have $CFL_{mod}^2 \approx  1.3$. \label{factorfig}
}
\end{figure} 

\section{Error analysis \label{erreur}}

\subsection{Truncation error \label{error}} 

In this section, we analyze the truncation error. Finite difference scheme \eqref{FDlinear} is consistent with the linearized partial differential equation if for any smooth function $\phi(t,x)$ the local error $E_{\delta t, \delta x}$ satisfies 
\begin{equation}
E_{\delta t, \delta x} = P\phi-P_{\delta t,\delta x}^{i} \phi \rightarrow 0,
\label{error}
\end{equation}  
as  $\delta t,\delta x\rightarrow 0$ with
\begin{equation*}
  \left\{
      \begin{aligned}
P\phi = \phi_{t}+ v \,\phi_{x}- \epsilon \, \phi_{xx} + \eta \, \I[\phi],\\
P_{\delta t,\delta x }^{i}\phi =\frac{\phi_{j}^{n+1}-\phi_{j}^{n}}{\delta t}+ v \, \frac{\phi_{j}^{n}-\phi_{j-1}^{n}}{\delta x}-\epsilon \, \frac{\phi_{j+1}^{n}-2\phi_{j}^{n}+\phi_{j-1}^{n}}{\delta x^{2}} + \eta \, \I_{\delta x}^{i}[\phi],
      \end{aligned}
    \right.
\end{equation*}
for $i=1,2$. \\

\begin{remark}
The practical implementation of the schemes requires to make some truncations. First, we consider a bounded domain $[0, T] \times [0, D]$ and to simplify, we also assume that $\delta t = T /N_{\delta t}$ and $\delta x = D/N_{\delta x}$ for some integers $N_{\delta x}$ and $N_{\delta t}$. Another truncation concerns                     
the integral operator for the nonlocal term $\I$. We replace $\int_{0}^{+\infty}$ with $\int_{0}^A$ and in the finite difference approximations \eqref{discretization14} and \eqref{discretization24} series $\sum_{l=1}^{\infty} $ are replaced with partial sums   $\sum_{l=1}^{A_{\delta x}}$, where $A=A_{\delta x} \,\delta x$. However, the truncation parameter $A$ has to be chosen judiciously. A ``short memory'' principle has been investigated to choose this parameter.
This principle is based on the fact that terms $l^{-7/3}$ and $l^{-1/3}$ in discretizations \eqref{discretization14} and \eqref{discretization24} decrease with $l$ therefore, we have to take into account the behaviour of $\varphi(x)$ only in the recent past, i.e. in the interval $[x-L,x]$, where $L>0$ is called the ``memory length''. Finally, the use of the short-memory principle leads to the simple replacement of  $\sum_{l=1}^{+\infty}$ by  $\sum_{l=1}^{A_{\delta x}}$, where $A_{\delta x}= [\frac{L}{\delta x}]$ \cite{pod}. \\
Note that the truncation parameter $A$ also strongly depends on the discretization of the nonlocal term $\I$ because  $l^{-7/3}$ decreases more quickly than $l^{-1/3}$. For the sake of simplicity, we will denote by $A$ the truncation parameter for the discretizations  \eqref{discretization14} and \eqref{discretization24}. 
\end{remark} 

\bigskip 

\begin{proposition}[Local error\label{consistance}] 

The local error of the numerical scheme \eqref{FDlinear} satisfies: 
\begin{itemize}
\item For $\I_{\delta x}^1$:
\begin{equation}
|E_{\delta t, \delta x, A}^1|  \leq \mathcal{O}(\delta t) +\mathcal{O}(\delta x^{2/3}) + \mathcal{O}(A^{-1/3}) +  \mathcal{O}\left(  A^{-1/3} \, \delta x \right) + \mathcal{O}\left(  A^{2/3} \, \delta x^2 \right) .
\label{errorestimate031}
\end{equation}

\item For $\I_{\delta x}^2$:
\begin{equation}
|E_{\delta t, \delta x, A}^2|  \leq \mathcal{O}(\delta t) +\mathcal{O}(\delta x^{2/3}) + \mathcal{O}(A^{-1/3}) + \mathcal{O}\left( A^{-4/3} \, \delta x \right) + \mathcal{O}\left( A^2 \, \delta x^3 \right) + \mathcal{O}\left( A \, \delta x^2 \right).
\label{errorestimate032}
\end{equation}
\end{itemize}

\end{proposition}

\begin{proof}
From Taylor series, we have
\begin{equation}
\phi_{t}\left(t_{n},x_{i}\right)- \frac{\phi^{n+1}_i -\phi^{n}_{i}}{\delta t} = \mathcal{O}(\delta t), 
\label{DL1}
\end{equation}
\begin{equation}
\phi_{x}\left(t_{n},x_{i}\right)- \frac{\phi^{n}_i -\phi^{n}_{i-1}}{\delta x} = \mathcal{O}(\delta x), 
\label{DL2}
\end{equation}
\begin{equation}
\phi_{xx}\left(t_{n},x_{i}\right)- \frac{\phi^{n}_{i+1} -2\phi^{n}_{i} + \phi^{n}_{i-1}}{\delta x^2} = \mathcal{O}(\delta x^2).
\label{DL3}
\end{equation}
Let us now study the truncation error for the nonlocal term $\I$. 

\noindent \textbf{For $\I_{\delta x}^1$:}
We rewrite \eqref{nonlocalterm3b} as follows
\begin{eqnarray}
&\,& \I[\phi(t^n, \cdot)](x_i)  =  \sum_{j=1}^{A_{\delta x} } \int_{(j-1/2)\delta x}^{(j+1/2)\delta x} \xi^{-1/3} \phi_{xx}(t^n, x_i- \xi) \, d\xi + \int_{0}^{\frac{\delta x}{2}} \xi^{-1/3} \phi_{xx}(t^n, x_i- \xi) \, d\xi \nonumber \\ 
& \, & \hspace{2.5 cm} + \int_{ A + \frac{\delta x}{2} }^{ A } \xi^{-1/3} \phi_{xx}(t^n, x_i- \xi) \, d\xi + \int_A^{+\infty} \xi^{-1/3} \phi_{xx}(t^n, x_i- \xi) \, d\xi,
\label{intA}
\end{eqnarray}
and the discretization \eqref{discretization14} becomes
\begin{eqnarray*}
\I_{\delta x}^1 [\phi(t^n,\cdot)]_i := \sum_{j = 1}^{A_{\delta x}} \delta x \, \xi_{\delta x}^{-1/3} \, \Phi^n_{\delta x}  =  \sum_{j = 1}^{A_{\delta x}} \int_{(j-1/2)\delta x}^{(j+1/2)\delta x} \xi_{\delta x}^{-1/3} \, \Phi^n_{\delta x} d\xi,
\end{eqnarray*}
with $\xi_{\delta x} := j \delta x$ and $\Phi^n_{\delta x}= \frac{\phi^n_{i-j+1}- 2 \phi^n_{i-j} + \phi^n_{i - j - 1}}{\delta x^2}  $. 
Using \eqref{intA}, we then get the following relation:
\begin{eqnarray*}
\I^1_{\delta x}[\phi(t^n,\cdot)]_i - \I[\phi(t^n,\cdot)](x_i) & = & \sum_{j = 1}^{A_{\delta x}} \int_{(j-1/2)\delta x}^{(j+1/2)\delta x} \left( \xi_{\delta x}^{-1/3} \, \Phi^n_{\delta x} - \xi^{-1/3} \phi_{xx}(t^n, x_i- \xi) \,  \right) \, d\xi \\
&-&  \int_{0}^{\frac{\delta x}{2}} \xi^{-1/3} \phi_{xx}(t^n, x_i- \xi) \, d\xi + \int_{ A  }^{ A + \frac{\delta x}{2}} \xi^{-1/3} \phi_{xx}(t^n, x_i- \xi) \, d\xi \\ 
&-& \int_A^{+\infty} \xi^{-1/3} \phi_{xx}(t^n, x_i- \xi) \, d\xi, \\
& = & T_1 - T_2 + T_3 -T_4.  
\end{eqnarray*}
Let us study the term $T_1$. Since 
\begin{eqnarray*}
 \xi_{\delta x}^{-1/3} \, \Phi^n_{\delta x} - \xi^{-1/3} \phi_{xx}(t^n, x_i- \xi) &=& \left(\xi_{\delta x}^{-1/3} - \xi^{-1/3} \right) \phi_{xx}(t^n, x_i- \xi) \\
&+&  \xi_{\delta x}^{-1/3} \left( \Phi^n_{\delta x} -\phi_{xx}(t^n, x_i- \xi) \right), 
\end{eqnarray*}
then
\begin{eqnarray*}
T_1 & = & \sum_{j = 1}^{A_{\delta x}} \int_{(j-1/2)\delta x}^{(j+1/2)\delta x} \left(\xi_{\delta x}^{-1/3} - \xi^{-1/3} \right) \phi_{xx}(t^n, x_i- \xi) \, d\xi \\ 
& \, & \hspace{2.5 cm} + \sum_{j = 1}^{A_{\delta x}} \int_{(j-1/2)\delta x}^{(j+1/2)\delta x} \xi_{\delta x}^{-1/3} \left( \Phi_{\delta x}^n -\phi_{xx}(t^n, x_i- \xi) \right) \, d\xi, \\
& = & T_{1,1} + T_{1,2} . 
\end{eqnarray*}
By the mean value theorem applied to  $z \rightarrow |z|^{-1/3}$, we have for all $\xi \in [(j-\frac{1}{2}) \delta x; (j+ \frac{1}{2})\delta x]$
\begin{eqnarray*}
|\xi_{\delta x}^{-1/3} - \xi^{-1/3}| & \leq & \sup_{z \in [(j-\frac{1}{2}) \delta x; (j+ \frac{1}{2})\delta x]} |\frac{1}{3} z^{-4/3}| \, |\xi_{\delta x} - \xi |, \\
&\leq& \frac{1}{3} |(j-\frac{1}{2}) \delta x|^{-4/3} |\xi_{\delta x} - \xi|,  \\
& \leq &  \frac{1}{6} |(j-\frac{1}{2}) \delta x|^{-4/3} \delta x. 
\end{eqnarray*}
Thus, integrating over $[(j-\frac{1}{2}) \delta x; (j+ \frac{1}{2})\delta x]$ we get
\begin{eqnarray}
|T_{1,1}| & \leq & C \, \delta x^{2/3}  \sum_{j = 1}^{A_{\delta x}}    \frac{1}{(j-1/2)^{4/3}} \leq C \, \delta x^{2/3},
\label{T11a}
\end{eqnarray}
because $\sum_{j \geq 1}  \frac{1}{(j-1/2)^{4/3}} < + \infty$ and 
$C$ is a positive constant which depends on $||\phi_{xx}||_{L^{\infty}((0,T) \times \R)}$. \\
Moreover, by classical midpoint quadrature rule 
\begin{equation}
\int_{(j-1/2)\delta x}^{(j+1/2)\delta x} \phi_{xx}(t^n, x_i - \xi) \, d\xi = \delta x \, \phi_{xx}(t^n, x_i - \xi_{\delta x}) + \frac{\delta x^3}{24} \phi_{4x} (t^n, x_i - \eta_j), 
\label{formulaerreur}
\end{equation}
with $\eta_j \in [(j-1/2)\delta x; (j+1/2)\delta x]$ then
\begin{eqnarray*}
T_{1,2} & = & \sum_{j = 1}^{A_{\delta x}} (j\delta x)^{-1/3} \left[\delta x \Phi_{\delta x}^{n} - \int_{(j-1/2)\delta x}^{(j+1/2)\delta x} \phi_{xx}(t^n, x_i - \xi) \, d\xi \right], \\
&=& \sum_{j = 1}^{A_{\delta x}} (j\delta x)^{-1/3} \left[ \delta x \Phi_{\delta x}^{n} -  \delta x \, \phi_{xx}(t^n, x_i - \xi_{\delta x}) - \frac{\delta x^3}{24} \phi_{4x} (t^n, x_i - \eta_j)\right].
\end{eqnarray*}
From Taylor series, we have
\begin{equation}
\phi_{xx}(t^n, x_i - \xi_{\delta x} ) = \Phi^n_{\delta x} - \frac{\delta x^2}{12} \phi_{4x}(t^n, x_i - \xi_{\delta x}) + \mathcal{O}(\delta x^4),
\end{equation}
thus we obtain 
\begin{eqnarray*}
T_{1,2} &=& \sum_{j = 1}^{A_{\delta x}} (j \delta x)^{-1/3}\left\lbrace \frac{\delta x^3}{12} \, \phi_{4x} (t^n, x_i - \xi_{\delta x}) - \frac{\delta x^3}{24} \phi_{4x} (t^n, x_i - \eta_j) + \mathcal{O}(\delta x^5) \right\rbrace. 
\end{eqnarray*}
We finally get
\begin{eqnarray*}
|T_{1,2}|& \leq & \mathcal{O}( \delta x^{8/3})  \sum_{j = 1}^{A_{\delta x}} \frac{1}{j^{1/3}} =   \mathcal{O}(\delta x^{8/3}) \left( 1 + \sum_{j = 2}^{A_{\delta x}} \frac{1}{j^{1/3}} \right), \nonumber\\
&=& \mathcal{O}(\delta x^{8/3})\left( 1 +\int_0^{A_{\delta x}} y^{-1/3} \,dy \right) , 
\end{eqnarray*}
which implies that 
\begin{equation}
|T_{1,2} | \leq \mathcal{O}(\delta x^{8/3}) + \mathcal{O}(A^{2/3} \, \delta x^{2}),
\label{T11b} 
\end{equation}
because $A_{\delta x} = \frac{A}{\delta x}$. \\
We next control the term $T_2$ by
\begin{eqnarray}
|T_2| &\leq & ||\phi_{xx}||_{L^{\infty}( (0,T) \times \R )} \int_{0}^{\frac{\delta x}{2}} \, |\xi|^{-1/3} \, d\xi  = C \delta x^{2/3}.
\label{T12}
\end{eqnarray}
We estimate $T_3$ as follows:
\begin{eqnarray}
|T_3| &\leq & ||\phi_{xx}||_{L^{\infty}\left((0,T) \times \R \right) } \int_{A }^{ A + \delta x} |\xi|^{-1/3} \, d\xi, \nonumber  \\
&\leq & ||\phi_{xx}||_{L^{\infty}\left((0,T) \times \R \right)} \, A^{-1/3} \, \delta x. 
\label{T13}
\end{eqnarray}
Finally, using an integration by parts, the term $T_4$ is written as
\begin{eqnarray*}
T_4 &=& \int_A^{+\infty} \xi^{-1/3} \phi_{xx}(t^n,x_i- \xi) \, d\xi, \\
&=& -A^{-1/3} \phi_x(t^n,x_i-A) + \frac{1}{3} \int_A^{+\infty} \xi^{-4/3} \phi_x(t^n, x_i-\xi) \, d\xi,  
\end{eqnarray*}
hence, we obtain
\begin{equation}
|T_4| \leq C A^{-1/3},
\label{T14}
\end{equation}
where $C$ is a positive constant which depends on $||\phi_{x}||_{L^\infty((0,T)\times \R)}$.\\
Hence, using relations \eqref{DL1}, \eqref{DL2}, \eqref{DL3}, \eqref{T11a}, \eqref{T11b}, \eqref{T12}, \eqref{T13} and  \eqref{T14}, we obtain
\begin{eqnarray*}
|E_{\delta x, \delta t, A}^1| = |P\phi(t_n,x_i)-P_{\delta t,\delta x}^{1} \phi | & \leq & \mathcal{O}(\delta x^{2/3}) + \mathcal{O}(\delta t) + \mathcal{O}(A^{-1/3}) \\ 
&\,& + \, \mathcal{O}(A^{2/3} \, \delta x^{2}) + \mathcal{O}\left( A^{-1/3} \, \delta x \right),
\end{eqnarray*}
which completes the proof for $\I_{\delta x}^1$. 



\noindent \textbf{For $\I_{\delta x}^2$:} 
As previously, we rewrite \eqref{intformula} as
\begin{eqnarray*}
\I[\phi(t^n, \cdot)](x_i) & = & \sum_{j=1}^{A_{\delta x} } \int_{(j-1/2)\delta x}^{(j+1/2)\delta x} \Phi(t^n,\xi) \, |\xi|^{-7/3} \, d\xi + \int_{0}^{\frac{\delta x}{2}} \Phi(t^n,\xi) \, |\xi|^{-7/3} \, d\xi \\ 
& + &  \int_{ A + \frac{\delta x}{2} }^{ A } \Phi(t^n,\xi) \, |\xi|^{-7/3} \, d\xi + \int_A^{+\infty}  \Phi(t^n,\xi) \, |\xi|^{-7/3} \, d\xi,
\end{eqnarray*}
with $ \Phi(t^n,\xi) = \frac{4}{9} \left( \phi(t^n, x_i-\xi) -  \phi(t^n, x_i) + \phi_{x}(t^n,x_i) \, \xi \right) $
and the approximated integral \eqref{discretization24} becomes
\begin{eqnarray*}
\I_{\delta x}^2 [\phi(t^n,\cdot)]_i 
& := & \sum_{j = 1}^{A_{\delta x}} \int_{(j-1/2)\delta x}^{(j+1/2)\delta x} \xi_{\delta x}^{-7/3} \, \Phi_{\delta x}^n d\xi,
\end{eqnarray*}
with $\Phi^n_{\delta x} = \frac{4}{9} \left( \phi^n_{i-j} - \phi^n_i + \frac{\phi^n_{i+1} -\phi^n_{i-1}}{2}j \right) $ and $\xi_{\delta x} = j \delta x$. \\
Let us now estimate the error on the nonlocal term. 
\begin{eqnarray*}
\I^2_{\delta x}[\phi(t^n, \cdot)]_i - \I[\phi(t^n, \cdot)](x_i) &=& \sum_{j = 1}^{A_{\delta x}} \int_{(j-1/2)\delta x}^{(j+1/2)\delta x} \left( \Phi_{\delta x}^n \xi_{\delta x}^{-7/3} - \Phi(t^n,\xi)|\xi|^{-7/3} \right) \, d\xi  \\
&-&  \int_{0}^{\frac{\delta x}{2}} \Phi(t^n,\xi) \, |\xi|^{-7/3} \, d\xi + \int_{ A  }^{ A + \frac{\delta x}{2}} \Phi(t^n,\xi) \, |\xi|^{-7/3} \, d\xi \\
&-&  \int_A^{+\infty} \Phi(t^n,\xi) \, |\xi|^{-7/3} \, d\xi, \\
&=& T_1 - T_2 + T_3 - T_4. 
\end{eqnarray*}
Let us study the term $T_1$. As previously for $\I^1_{\delta x}$, we rewrite $T_1$ as
\begin{eqnarray*}
T_1 &=& \sum_{j = 1}^{A_{\delta x}} \int_{(j-1/2)\delta x}^{(j+1/2)\delta x} (\xi_{\delta x}^{-7/3}- |\xi|^{-7/3}) \Phi(t^n,\xi) + \xi_{\delta x}^{-7/3}(\Phi^n_{\delta x} - \Phi(t^n,\xi)) \, d\xi, \\
&=& T_{1,1} + T_{1,2}.  
\end{eqnarray*}
By the mean value theorem applied to $z \rightarrow |z|^{-7/3}$, we have for all $\xi \in [(j-1/2)\delta x; (j+1/2)\delta x]$
\begin{eqnarray*}
| \xi_{\delta x}^{-7/3} - \xi^{-7/3} | &\leq & \sup_{z \in [(j-1/2)\delta x; (j+1/2)\delta x]} |\frac{7}{3} z^{-10/3}| \, |\xi_{\delta x} - \xi |, \\
&\leq &  \frac{7}{6} \delta x  \frac{1 }{| (j-1/2)\delta x |^{10/3}},  \\
&=& \frac{7}{6} \delta x^{-7/3}  \frac{1 }{(j-1/2)^{10/3}}.
\end{eqnarray*}
Next, by Taylor-Lagrange formula, we have
\begin{eqnarray*}
|T_{1,1}| &=& \frac{4}{9} \sum_{j=1}^{A_{\delta x}} \int_{(j-1/2)\delta x}^{(j+1/2)\delta x} |\xi_{\delta x}^{-7/3}- |\xi|^{-7/3}| \left| \phi(t^n, x_i-\xi) -  \phi(t^n, x_i) + \phi_{x}(t^n,x_i) \, \xi  \right|, \\
&\leq & C \sum_{j=1}^{A_{\delta x}} \int_{(j-1/2)\delta x}^{(j+1/2)\delta x}  \delta x^{-7/3} \frac{1}{(j-1/2)^{10/3}} ||\phi_{xx}||_{L^{\infty}((0,T)\times \R)} \frac{\xi^2}{2} \, d\xi, \\
&\leq & C \delta x^{-7/3} \sum_{j=1}^{A_{\delta x}} \frac{1}{(j-1/2)^{10/3}} \int_{(j-1/2)\delta x}^{(j+1/2)\delta x} \xi^2 \, d\xi, \\
&\leq & C \delta x^{-7/3} \sum_{j=1}^{A_{\delta x}} \frac{1}{(j-1/2)^{10/3}} \left(  \delta x^3 j^2 + \delta x^3 \right), \\
&\leq &  C \delta x^{2/3},
\end{eqnarray*}
because $ \sum_{j=1}^{\infty} \frac{j^2}{(j-1/2)^{10/3}} < \infty$. $C$ denotes a positive constant which depends on $||\phi_{xx}||_{L^{\infty}((0,T)\times \R)}$  and may vary from line to line. \\
Moreover, using again midpoint quadrature rule, we have
\begin{equation*}
\int_{(j-1/2)\delta x}^{(j+1/2)\delta x} \Phi(t^n, \xi) \, d\xi = \delta x \, \Phi(t^n, j\delta x) + \frac{\delta x^3}{24} \Phi_{xx}(t^n, \eta_j ),
\end{equation*}
with $\eta_j \in [(j-\frac{1}{2})\delta x, (j+\frac{1}{2})\delta x]$. Hence,
\begin{equation*}
T_{1,2} =  \sum_{j=1}^{A_{\delta x}} \xi_{\delta x}^{-7/3} \left[  \delta x \, \Phi^n_{\delta x} - \delta x \, \Phi(t^n, j\delta x) - \frac{\delta x^3}{24} \Phi_{xx}(t^n, \eta_j ) \right]. 
\end{equation*}
But using again Taylor expansion, we get
\begin{equation*}
\phi_x(t^n, x_i) = \frac{\phi(t^n,x_{i+1})-\phi(t^n,x_{i-1})  }{2 \delta x} - \frac{\delta x^2}{6} \phi_{3x} (t^n, x_i) + \mathcal{O}(\delta x^3),
\end{equation*}
and so 
\begin{eqnarray*}
\Phi^n_{\delta x} - \Phi(t^n,j \delta x) & = & \frac{4}{9} \left[ \phi^n_{i-j}- \phi^n_i + \frac{\phi^n_{i+1}-\phi^n_{i-1}}{2}j - \phi(t^n,x_i-j\delta x)+ \phi(t^n,x_i)  - \phi_x(t^n,x_i) j \delta x  \right],    \\
&=& \frac{4}{9} \left[ \xi_{\delta x} \frac{\delta x^2}{6} \phi_{3 x}(t^n, x_i) + j \mathcal{O}(\delta x^4)\right]. 
\end{eqnarray*}
Thus,
\begin{eqnarray*}
T_{1,2} & = & C \sum_{j=1}^{A_{\delta x}} \xi_{\delta x}^{-7/3} \left[  \xi_{\delta x}  \frac{\delta x^3}{6} \phi_{3 x}(t^n, x_i) + j \mathcal{O}(\delta x^5) - \frac{\delta x^3}{24} \Phi_{xx}(t^n, \eta_j) \right],   \\
& =& C \sum_{j=1}^{A_{\delta x}} \xi_{\delta x}^{-7/3} \left[ \xi_{\delta x}  \frac{\delta x^3}{6} \phi_{3 x}(t^n, x_i) + j \mathcal{O}(\delta x^5) - \frac{4}{9}\frac{\delta x^3}{24}  \phi_{xx}(t^n, x_i-\eta_j) \right].
\end{eqnarray*}
We finally get
\begin{eqnarray*}
 | T_{1,2}| &\leq& C \sum_{j=1}^{A_{\delta x}} \left(j^{-4/3}\delta x^{5/3} + j \mathcal{O}(\delta x^5)+  \delta x^3 \right), \\
  &= & \mathcal{O}(\delta x^{5/3}) + \mathcal{O}\left( A^2 \, \delta x^3 \right) + \mathcal{O}\left( A \, \delta x^2 \right),
\end{eqnarray*}
where $C$ is a positive constant which depends on $||\phi_{3x}||_{L^{\infty}((0,T) \times \R)} $ and $||\phi_{xx}||_{L^{\infty}((0,T) \times \R)} $.\\
Let us now study $T_2$. Using Taylor-Lagrange formula, we have
\begin{eqnarray*}
|T_2| &\leq& C ||\phi_{xx}||_{L^{\infty}\left( (0,T) \times \R \right) } \int_0^{\frac{\delta x}{2}} \frac{\xi^2}{|\xi|^{7/3}} \, d\xi, \\
&\leq & \mathcal{O}(\delta x^{2/3}),
\end{eqnarray*}
where $C$ is a positive constant.\\ 
Let us next consider $T_3$: 
\begin{eqnarray*}
T_3 &= &\int_{A}^{A + \frac{\delta x}{2} } |\xi|^{-7/3} \Phi(t^n, \xi) \, d\xi, \\
&=&\frac{4}{9} \int_{A}^{A + \frac{\delta x}{2} } \left[   \phi(t^n, x_i-\xi) - \phi(t^n, x_i) + \phi_x(t^n, x_i) \xi \right] |\xi|^{-7/3} \, d\xi. 
\end{eqnarray*}
Then
\begin{eqnarray*}
|T_3| &\leq & C \int_{A}^{A + \frac{\delta x}{2} } (|\xi|^{-7/3} + |\xi|^{-4/3} )\, d\xi, \\
&\leq & C\left(A^{-7/3} + A^{-4/3} \right) \delta x,   \\
&\leq & \mathcal{O}\left( \delta x \, A^{-4/3} \right),
\end{eqnarray*}
with $C$ a positive constant which depends on $||\phi||_{L^{\infty}\left( (0,T) \times \R \right) } $ and $||\phi_{x}||_{L^{\infty}\left( (0,T) \times \R \right) } $.\\ 
And since
\begin{eqnarray*}
|T_4| &\leq & \frac{4}{9} \int_A^{+ \infty} |\phi(t^n, x_i - \xi) - \phi(t^n, x_i) | \, |\xi|^{-7/3} \, d\xi + \frac{4}{9}   \int_A^{+ \infty} |\phi_x(t^n, x_i ) | \, |\xi|^{-4/3} \, d\xi, \\
&\leq& C A^{-1/3},
\end{eqnarray*}
where $C$ is a positive constant which depends on $||\phi||_{L^{\infty}((0,T)\times \R)}$ and $||\phi_x||_{L^{\infty}((0,T)\times \R)}$, we finally get
\begin{eqnarray*}
|E_{\delta x, \delta t, A}^2|:= |P\phi(t_n,x_i)-P_{\delta t,\delta x}^{2} \phi |  &\leq &  \mathcal{O}(\delta t) +\mathcal{O}(\delta x^{2/3}) + \mathcal{O}(A^{-1/3}) + \mathcal{O}\left( A^{-4/3} \, \delta x \right) \\
&+&   \mathcal{O}\left( A^2 \, \delta x^3 \right) + \mathcal{O}\left( A \, \delta x^2 \right).
\end{eqnarray*}
The proof of this proposition is now completed. 
\end{proof}


\begin{remark}
From previous Proposition, we can see that the numerical scheme \eqref{FDlinear}  with $\I_{\delta x}^1$ (resp. $\I_{\delta x}^2$) is consistent if $\delta x << A^{-1/3}$ (resp. $\delta x << A^{-2/3}$). 
\end{remark}

\subsection{Convergence experiments.}

In this section, we investigate the convergence using numerical simulations. Despite much effort we are unable to prove theoretically the convergence of the numerical solution towards the exact continuous solution. Indeed, the Lax procedure ``stability $+$ consistence $=$ convergence'' cannot be applied here due to the instability of low frequencies.    
\\
In what follows, $l^{1}$-norm is used to measure the accuracy of approximated solutions. Thus, we analyze the following error 
\begin{equation}
E_{1} = \frac{1}{N}  \sum_{n=0}^{N}   (| u_j^1(T) - u_j^2(T)|),
\end{equation}
where $u^1, u^2$ are, respectively, computed for space steps $\delta x/2$ and $\delta x/4$, until a final time $T$.  \\
Figure \ref{cvx} shows the numerical convergence rates obtained with the initial data displayed in Figure \ref{cinitiale3b}. These rates were obtained using $\delta x = 10^{-1}, 10^{-2}, 10^{-3}, 10^{-4}$. We plot the logarithm of the error $E_1$ versus the logarithm of $\delta x$. 

\begin{figure}[h!]
	\centering
	\includegraphics[scale=0.6]{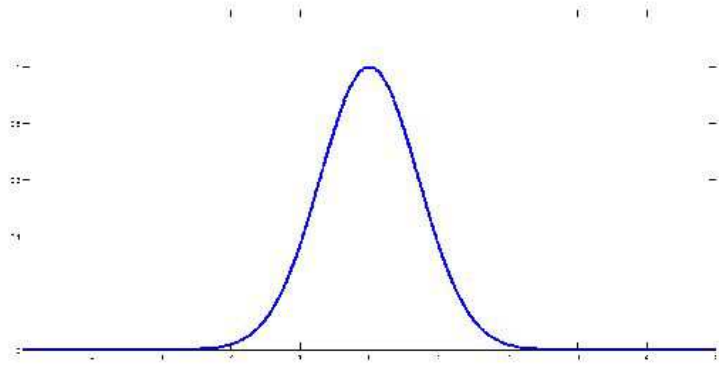} 
	\caption{Initial data used for numerical experiments.}
	\label{cinitiale3b}	
\end{figure}

\begin{figure}[h!]
	\centering
	\includegraphics[scale=0.7]{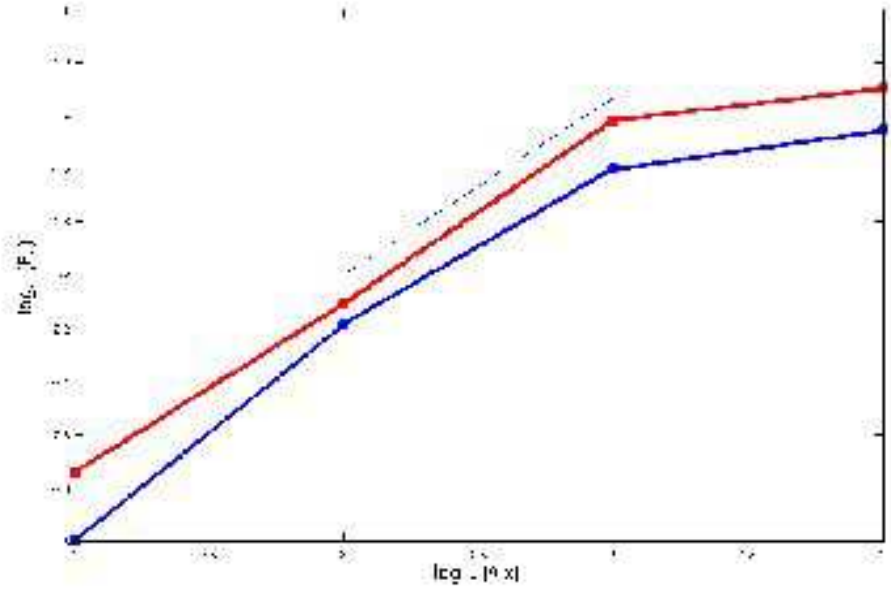} 
	\caption{Convergence in space for $\I_{\delta x}^1$ (red) and $\I_{\delta x}^2$ (blue). Dotted line has slope $2/3$.}
	\label{cvx}	
\end{figure}

\subsection{Phase error}
Numerical schemes produce, according to cases, results ahead of or delayed w.r.t exact solutions. In this section,
we are interested in the error made on the velocity introduced by the discretization. Let us first note that, in addition to the anti-diffusive effect, the nonlocal term is also responsible of the motion of the initial data. Indeed, we saw in Section \ref{preliminaire} that the continuous amplification factor has an imaginary part $e^{- i  \left( \eta \frac{\sqrt{3}}{2} \Gamma(\frac{2}{3}) \,  k \vert k \vert^{1/3} +  \, v  k \right) \delta t}.$
 Therefore, the advection term is not the unique factor of displacement. \\
It is the error on the argument $-\delta t( v\, k + \frac{\sqrt{3}}{2}\Gamma(\frac{2}{3}) \eta \, k \vert k \vert^{1/3})$ that causes the phase error. To evaluate this error, we rewrite the discrete amplification factor $g_j$ introduced in Section \ref{stability} as
\begin{equation*}
g_j = |g_j| e^{-i \theta_{d_j}},
\end{equation*}
for $j=1,2$, where $\theta_{d_j}$ is the argument of the discrete amplification factor $g_j$. The phase lag during one time step is then given by
\begin{equation*}
E_j =   (v \, k \, + \, \frac{\sqrt{3}}{2}\Gamma(\frac{2}{3}) \eta  \, k \, |k|^{1/3}) \delta t  -\theta_{d_j} .
\end{equation*}   
Thus, if $E_j$ is positive, the numerical wave goes slower than the physical wave and it goes faster if $E_j$ is negative. 
We have computed, in Tables \ref{table13}, \ref{table23} and  \ref{table33}, the phase delay after one oscillation 
$$\Delta_j = 1- \frac{\theta_{d_j}}{ (v \, k \, + \, \frac{\sqrt{3}}{2}\Gamma(\frac{2}{3}) \eta  \, k \, |k|^{1/3}) \delta t }, \mbox{ for } j=1,2, $$  
for different values of $Cr = \frac{v \, \delta t}{\delta x}$, $D_f = \frac{2 \, \epsilon \, \delta t}{\delta x^2}$ and $F_o = \frac{\eta \, \delta t}{\delta x^{4/3}} $. \\
We remark that scheme with the discretization $\I_{\delta x}^2$ involves a delay larger than the model with $\I_{\delta x}^1$. 
\subsection{Discrete amplification vs. continuous amplification factors}

Let us define  
$$G_1 = \frac{|g_1 |}{|G_{cont}|}, \, G_2 = \frac{|g_2 |}{|G_{cont}|},$$
for different values of $Cr = \frac{v \, \delta t}{\delta x}$, $D_f = \frac{2 \, \epsilon \, \delta t}{\delta x^2}$ and $F_o = \frac{\eta \, \delta t}{\delta x^{4/3}} $. 
Results are reported in Tables \ref{table13}, \ref{table23} and  \ref{table33}. 
We have $CFL_{mod}^i = Cr + Df + \lambda_i F_o$,  where $\lambda_1 = 2-2^{-1/3}, \lambda_2 = \frac{4}{9} \left(\zeta(\frac{4}{3})-1 \right).  $
When  the condition $CFL_{mod}^i >1$ is violated because
$Cr$ or  $D_f$ are close to one, we note that   high frequencies $\theta > \pi/2$ are more amplified by the discrete schemes than  by the exact continuous problem.  
Whereas when  the condition $CFL_{mod}^i >1$ is violated because
$F_o$ is close to one,  high frequencies $\theta > \pi/2$ are less amplified by the discrete schemes than by the continuous problem, as can be checked in Table \ref{table33}. This is one unexpected benefit of the discretization scheme in the unfavourable case where nonlocal anti-diffusion  is predominant. If we take a closer look at Table \ref{table33} we notice that $|g_2|$ may be greater than one even if $CFL_{mod}^2<1$. This is due to the fact that $\eta$ being big, the stability threshold frequency $\theta_0$ is close to $\pi$, the aliasing limit frequency.    

\begin{table}[htbp]
\begin{center}
\begin{tabular}{*{8}{c}}
  \hline
  $Cr$ & $CFL_{mod}^1$ & $CFL_{mod}^2$        & $\theta$ &   $\Delta_1$      & $\Delta_2$ & $G_1$ & $G_2$   \\
  \hline
  0.2 &  0.5206        &    0.5156           & $\pi/6$  & 0.0082    & 0.0333   &  0.9584   & 0.9788 \\
      &                &                     & $ \pi/4$ & 0.0024    & 0.0573   &  0.9102   & 0.9550 \\
      &                &                     & $\pi/2$  & -0.0715   & 0.1610   &  0.6824   & 0.8394 \\
      &                &                     & $3\pi/4$ & -0.2684   & 0.3911   & 0.3788    & 0.6626 \\
      &                &                     &          &           &          &           &   \\
  0.5 &  0.8206        &   0.8156            & $\pi/6$   & -0.0128   & 0.0104   &  0.9541   & 0.9736   \\
      &                &                     &  $ \pi/4$ & -0.0433   & 0.0091   &  0.9048   & 0.9452   \\
      &                &                     & $\pi/2$   & -0.02476  & -0.0315  &  0.7439   & 0.8152   \\
      &                &                     &  $3\pi/4$ & -0.4733   & -0.1628  &  0.8284   & 0.6391 \\
      &                &                     &           &           &          &           &    \\
 0.9  &  1.2206        &  1.2156             &   $\pi/6$   & -0.0103  & 0.0107   &  0.9870   & 1.0047    \\
      &                &                     &  $ \pi/4$ & -0.0306   & 0.0128   &  0.9869   & 1.0182  \\
      &                &                     &  $\pi/2$   & -0.0781   & 0.0172   &  1.1776   & 1.1522 \\
      &                &                     &  $3\pi/4$ & -0.0210   & 0.0371   &  1.8187   & 1.5053 \\
      &                &                     &            &           &          &            &  \\    
  \hline
\end{tabular}
\end{center}
\caption{  
Dampening and phase error for $D_f =0.2$ and $F_o =0.1 $
\label{table13}}
\end{table}

\begin{table}[htbp]
\begin{center}
\begin{tabular}{*{8}{c}}
  \hline
  $D_f$  & $CFL_{mod}^1$ & $CFL_{mod}^2$  & $\theta$ & $\Delta_1$ & $\Delta_2$    & $G_1$  & $G_2$  \\
  \hline
  0.2    &   0.4206      &  0.4156         & $\pi/6$  & 0.0213      & 0.0481        &  0.9654 &  0.9859  \\
         &               &                 & $ \pi/4$ & 0.0303      & 0.0868        &  0.9250 &  0.9707   \\
         &               &                & $\pi/2$  & 0.0525      & 0.2562        &  0.7340 &  0.9057  \\
         &               &                &  $3\pi/4$ & 0.1721      & 0.5567        &  0.4834 &  0.8487  \\
         &               &                &           &             &               &         &       \\
  0.4    &  0.6206       &  0.6156        & $\pi/6$   & -0.0066     & 0.0216        &  0.9649 &  0.9860 \\
         &               &                &  $ \pi/4$ & -0.0358     & 0.0276        &  0.9216 &  0.9701 \\
         &               &                & $\pi/2$   & -0.3470     & 0.0154        &  0.6750 &  0.8841  \\
         &               &                &  $3\pi/4$ & -2.0019     & 0.0277        &  0.4153 &   0.7059 \\
         &               &                &           &             &               &         & \\
 0.8     &  1.0206       & 1.0156        &  $\pi/6$   & -0.0673    & -0.0361       &  0.9614 &  0.9837 \\
         &               &                & $ \pi/4$ & -0.1989     & -0.1169       &  0.9022 &   0.9568 \\
         &               &                & $\pi/2$   & -3.1203     & -1.4458       &  0.5305 & 0.6566\\
         &               &                &  $3\pi/4$ & -3.8370     & -3.6360       &  5.5254 & 3.5099\\
         &               &                 &           &             &               &         & \\    
  \hline
\end{tabular}
\end{center}
\caption{  
Dampening and phase error for $Cr =0.1$ and $F_o =0.1 $
\label{table23}}
\end{table}

\begin{table}[htbp]
\begin{center}
\begin{tabular}{*{11}{c}}
  \hline
  $F_o$   & $CFL_{mod}^1$      &  $CFL_{mod}^2$  & $\theta$  & $\Delta_1$  & $\Delta_2$ & $G_1$    & $G_2$  & $|g_1|$ &$|g_2|$ & $G_{cont}$ \\
  \hline
  0.2     &   0.5413           &   0.5312        & $\pi/6$   & 0.0307      &  0.0784    &  0.9455  & 0.9852  & 0.9741   & 1.0150  &   1.0302 \\
          &                    &                 &  $ \pi/4$ & 0.0363      &  0.1360    & 0.8852   & 0.9707  & 0.9180   & 1.007   &   1.0371   \\
          &                    &                 & $\pi/2$   & -0.0079     &  0.3495    &  0.6236  & 0.9052  &  0.6240  & 0.9057  &  1.0005 \\
          &                    &                 &  $3\pi/4$ & -0.1876     &  0.6374    &  0.3216  & 0.8182  &  0.2822  &  0.7180 &  0.8776 \\
          &                    &                 &  $\pi $   & -1.2548     &  1.0000    &  0.0574  & 0.6017  &  0.0574 
& 0.6017  &  0.6950  \\
          &                    &                 &           &             &            &          &         &          &         & \\
  0.5     &  0.9031            &  0.8780         & $\pi/6$   & 0.0591      & 0.1552     & 0.8992   & 0.9875  &   1.0092 & 1.1084  &  1.1224  \\
          &                    &                 &  $ \pi/4$ & 0.0650      & 0.2538     & 0.8043   & 0.9758  &  0.9664  & 1.1725  &  1.2016 \\
          &                    &                 & $\pi/2$   & -0.0103     & 0.5264     & 0.5235   & 0.8684  & 0.7589
& 1.2590  &  1.4498 \\
          &                    &                 &  $3\pi/4$ & -0.0906     & 0.7619     & 0.4215   & 0.6524  &  0.6992  & 1.0824  &  1.6590 \\
          &                    &                 &  $\pi $   &  -0.0430    & 1.0000     & 0.4202   & 0.5110  &  0.7435        
& 0.9042  &  1.7694  \\ 
          &                    &                 &           &             &            &          &         &        
&         & \\
 0.9      &  1.3857            &  1.3404         &  $\pi/6$  & 0.1064      & 0.2438     &  0.8602  &  0.9941 &  1.0822 
& 1.2512  &  1.2583   \\
          &                    &                 &  $ \pi/4$ & 0.1361      & 0.3752     &  0.7523  &  0.9762 &  1.0999  & 1.4273  &  1.4621 \\
          &                    &                 & $\pi/2$   & 0.2002      & 0.6556     &  0.5123  &  0.7449 & 1.2178 
& 1.7707  &  2.3771 \\
          &                    &                 &  $3\pi/4$ & 0.2981      & 0.8366     &  0.3873  &  0.4083 & 1.5020 
& 1.5836  &  3.8781 \\
          &                    &                 & $\pi $    & 0.3924      & 1.0000     & 0.2696   & 0.2126  & 1.6583 
& 1.3075  & 6.1519 \\          
          &                    &                 &           &             &            &          &         &   
&         &\\    
  \hline
\end{tabular}
\end{center}
\caption{  
Dampening and phase error for $Cr = 0.1$ and $D_f =0.2. $
\label{table33}}
\end{table}

\section{Concluding remarks \label{conclusion4}}

We have presented in this work a first investigation of finite differences schemes approximating the Fowler equation. We saw that the anti-diffusive behaviour of the nonlocal term does not enable to consider the classical notion of stability. 
Nevertheless, considering only the behaviour of the high frequencies (which should be quickly dampened), we exhibit numerical stability criteria which can be used to make simulations. Numerical computations have shown that numerical schemes dampened more than the continuous problem. Finally, consistency property has been proved and convergence of schemes has been investigated. \\



\bibliographystyle{plain}

\end{document}